\documentclass[11pt,letterpaper]{amsart}

\usepackage{amsmath}
\usepackage{amsthm}
\usepackage{amssymb}
\usepackage{amscd}
\usepackage{amsfonts}
\usepackage{amsbsy}
\usepackage[noadjust]{cite} %The noadjust option makes (\cite{ref}) space correctly on the left.
\usepackage{epsfig,afterpage}
\usepackage{paralist}
\usepackage{psfrag}
\usepackage{mathrsfs}       %Gives the script font enabled by the invoked command \mathscr
\usepackage{verbatim}
\usepackage{color}

\newtheorem {theorem} {Theorem}%[section]
\newtheorem {definition} [theorem]{Definition}
\newtheorem {proposition} [theorem]{Proposition}
\newtheorem {corollary} [theorem]{Corollary}
\newtheorem {lemma}  [theorem]{Lemma}
\newtheorem {remark} [theorem]{\sc Remark}

\title[Characterization of centers]{Characterization of centers by its complex separatrices}

\thanks{The authors are partially supported by the Agencia Estatal de Investigaci\'on grant number PID2020-113758GB-I00 and an AGAUR grant number 2021SGR-01618.}

\author[Isaac A. Garc\'ia and Jaume Gin\'e]{Isaac A. Garc\'ia and Jaume Gin\'e}

\address{Departament de Matem\`atica, Universitat de Lleida,
Avda. Jaume II, 69, 25001 Lleida, Spain}

\email{isaac.garcia@udl.cat}
\email{jaume.gine@udl.cat}



\begin{document}

\begin{abstract}
In this work we deal with analytic families of real planar vector fields $\mathcal{X}_\lambda$ having a monodromic singularity at the origin for any $\lambda \in \Lambda \subset \mathbb{R}^p$ and depending analytically on the parameters $\lambda$. There naturally appears the so-called center-focus problem which consists in describing the partition of $\Lambda$ induced by the centers and the foci at the origin. We give a characterization of the centers (degenerated or not) in terms of a specific integral of the cofactor associated to a real invariant analytic curve passing through the singularity, which always exists. Several consequences and applications are also stated.
\end{abstract}

\maketitle

\section{Introduction}

In this work we study families of real analytic planar autonomous differential systems
\begin{equation}\label{DCC-1}
\dot{x}= P(x,y; \lambda),  \ \ \dot{y}=Q(x,y; \lambda),
\end{equation}
where the dot denotes, as usual, derivative with respect to an independent time-variable and depending analytically on the parameters $\lambda \in \mathbb{R}^p$ of the family. Along the work we denote by $\mathcal{X} = P(x,y; \lambda) \partial_x + Q(x,y; \lambda) \partial_y$ the vector field associated to \eqref{DCC-1}. We are only interested in the restriction of family \eqref{DCC-1} to the monodromic parameter space $\Lambda \subset \mathbb{R}^p$ which is defined as the parameter subset for which $(x,y)=(0,0)$ is a {\it monodromic singularity}. In other words, the origin is a singular point of \eqref{DCC-1} where the structure of local flow turns around it for any $\lambda \in \Lambda$. The monodromic set $\Lambda$ is usually characterized by the blow-up process developed by Dumortier in \cite{Dum}, see also Arnold \cite{Ar}. In the papers \cite{AGR,AGR2} Algaba and co-authors present an algorithmic procedure that allows to specify the parameter conditions that defines $\Lambda$ in terms of the Newton diagram $\mathbf{N}(\mathcal{X})$ of the vector field $\mathcal{X}$, see also \cite{GGGrau}.

In this monodromic setting, I'lyashenko \cite{Ilyashenko} and \'Ecalle \cite{Ecalle} independently prove that the nature of the singularity at the origin only can be either a center or a focus due to the analyticity of  $\mathcal{X}$. A {\it center} is a singularity having a punctured neighborhood foliated by periodic orbits of $\mathcal{X}$ whereas in a {\it focus} the orbits spiral around it. Therefore there is induced a natural partition of $\Lambda$ into subsets according to these two possibilities and to characterize such partition is called the {\it center-focus problem} at the origin of $\mathcal{X}$.

There are several degrees of degeneracy in the center-focus problem according to the nature of the linear part of $\mathcal{X}$ at $(0,0)$. Thus in the {\it non-degenerate} monodromic singular points that have the eigenvalues of the Jacobian matrix $D \mathcal{X}(0,0)$ of $\mathcal{X}$ at $(0,0)$ given by non-zero pure imaginary eigenvalues, the solution of the center-focus problem goes back to the seminal works of Poincar\'e and Lyapunov, see for example the book \cite{RS} for a modern point of view. The next degree of degeneration appears when $D \mathcal{X}(0,0) \not\equiv 0$ but both eigenvalues are zero and this {\it nilpotent} center-focus problem was solved by Moussu \cite{Mo}. We refer to the {\it degenerate} center-focus problem in case that $\mathcal{X}$ has no linear part at $(0,0)$, that is, $D \mathcal{X}(0,0) \equiv 0$. There is no universal technique that discerns between a degenerated center or focus giving the stability of the origin although some partial results are presented under certain restrictions on $\mathcal{X}$, see for example \cite{Ga-Li-Ma-Ma,GaGi1,GGL,G,G1}. We are interested in this challenging open problem and this work is devoted to give a characterization of any type of center without any additional restriction.

The singularity at the origin can be desingularized by the blow-up procedure which transforms the singularity into a monodromic polycycle $\Gamma$. In this way $\Gamma$ has a Poincar\'e return map $\Pi: \Sigma \subset (\mathbb{R}^+, 0) \to (\mathbb{R}^+, 0)$ on one side of $\Gamma$. In the work \cite{Ilyashenko} Il'Yashenko shows that, in general, $\Pi$ is no longer differentiable at the origin but it is a semiregular map. Consequently  $\Pi$ has a Dulac asymptotic expansion possessing a linear leading term and $\Pi$ reduces to the identity map when the origin is a center of $\mathcal{X}$.

The simplest nature of $\Gamma$ corresponds to a periodic orbit, and this always happens when the monodromic singularity is non-degenerated. We can also reach this simple case for nilpotent singularities. Indeed there are also some degenerated monodromic singularities, called in this work the monodromic class ${\rm Mo}^{(p,q)}$, with associated periodic orbit $\Gamma$ because there are no characteristic directions after the first (weighted) polar blow-up. See the details in the forthcoming Definition \ref{Mo-pq}. A characteristic of the class ${\rm Mo}^{(p,q)}$ is that the Poincar\'e map $\Pi$ of $\Gamma$ is analytic which means that one can use the so-called Bautin method consisting in the classical procedure proposed by Poincar\'e and Lyapunov to approach the solution of the non-degenerated center-focus with minor modifications. The reader may consult for instance \cite{G0} and references therein.

In the work \cite{Me} of Medvedeva it is proved that, thanks to the analyticity of $\mathcal{X}_\lambda$ in $\lambda$, $\Pi$ has a Dulac asymptotic expansion of the form
\begin{equation}\label{Poinc-expan}
\Pi(x) = \eta_1 x + \sum_j P_j(\log x) x^{\nu_j},
\end{equation}
where $\eta_1 > 0$, the exponents $\nu_j > 1$ grow to infinity and the coefficients of the $P_j$ are polynomials whose coefficients depend analytically on the coefficients of $\mathcal{X}$. Shorting notation we have $\Pi(x) =
\eta_1 x + o(x)$ and the first (generalized) Poincar\'e-Lyapunov quantity $\eta_1$ can be computed following Medvedeva's work \cite{Me-2}. It is worth to mention here the works \cite{Ga-Li-Ma-Ma,Ga-Ma-Ma} that impose explicit constrains on $\mathcal{X}$ that allow to compute $\eta_1$, see also \cite{AGG3} and references therein.

The work is structured as follows. The main results are stated in section \ref{S1}. Section \ref{SS3} is dedicated first to the proof of Theorem \ref{teo-F-exists} and next to summarize well-known results on real analytic curves combined with a study of some relations arising between real analytic invariant curves of $\mathcal{X}$ and the invariant branches of $\mathcal{X}$. The weighted polar blow-up and its associated $(p,q)$--{\it characteristic directions} are explained in section \ref{S3}. In section \ref{S6} we prove Theorem \ref{Th-centers-pq-period} and Corollary \ref{Corol-semidef}. In section \ref{S7} we focus our attention on the vector fields in the monodromic class ${\rm Mo}^{(p,q)}$ and prove Theorem \ref{Th-centers-Mopq}. Later on in section \ref{S8} we present some nontrivial examples illustrating how the theory works. Finally in the last section \ref{S10} we show the big difficulty in dealing with the stability problem of degenerate monodromic singularities, namely to know the dominant term (in some sense) of the flow near the associated polycycle.

\section{Statement of the main results} \label{S1}

Hereinafter we omit the dependency on the parameters $\lambda$ in all the formulas except when it is relevant.
\newline

Now we state the first main result of the work whose proof, based on Camacho-Sad separatrix theorem, is given in section \ref{S-PT1}.

\begin{theorem}\label{teo-F-exists}
Let $\mathcal{X}$ be real analytic planar vector field having a monodromic singular point at the origin. Then there exists a real analytic invariant curve $F^{\mathbb{R}}(x,y)=0$ of $\mathcal{X}$ with $F^{\mathbb{R}}(0,0)=0$ and $F^{\mathbb{R}}$ having an isolated zero in $\mathbb{R}^2$ at the origin.
\end{theorem}

Before the statement of Theorem \ref{Th-centers-pq-period} we need to introduce some background. Given an analytic vector field $\mathcal{X} = P(x,y) \partial_x + Q(x,y) \partial_y$ with
$$
P(x,y) = \sum_{(i,j) \in \mathbb{N}^2} a_{ij} x^i y^{j-1}, \ \ Q(x,y) = \sum_{(i,j) \in \mathbb{N}^2} b_{ij} x^{i-1} y^{j},
$$
the support of $\mathcal{X}$ is defined by ${\rm supp}(\mathcal{X}) = \{ (i, j) \in \mathbb{N}^2 :  (a_{ij}, b_{ij}) \neq (0, 0)\}$. The boundary of the convex hull of the set
$$
\bigcup_{(i,j) \in {\rm supp}(\mathcal{X})} \{ (i,j) + \mathbb{R}_+^2 \}
$$
contains two open rays and a polygon. The Newton diagram $\mathbf{N}(\mathcal{X})$ of $\mathcal{X}$ is composed by that polygon together with the rays that do not lie on a coordinate axis in case they exist. Each edge of $\mathbf{N}(\mathcal{X})$ has endpoints in $\mathbb{N}^2$ and we associate to it the weights $(p,q) \in \mathbb{N}^2$, with $p$ and $q$ coprimes, given by the tangent $q/p$ of the angle between that segment and the ordinate axis. Throughout this work we denote by $W(\mathbf{N}(\mathcal{X})) \subset \mathbb{N}^2$ the set containing all the weights associated to the edges in $\mathbf{N}(\mathcal{X})$.
\newline

Given $(p,q) \in W(\mathbf{N}(\mathcal{X}))$, we perform the {\it weighted polar} blow-up $(x,y) \mapsto (\rho, \varphi)$ given by $(x,y) = \phi(\varphi, \rho) = (\rho^p \cos\varphi, \rho^q \sin\varphi)$ transforming \eqref{DCC-1} into the polar vector field $\dot{\rho} = R(\varphi, \rho)$, $\dot{\varphi} = \Theta(\varphi, \rho)$ and consider the differential equation
\begin{equation}\label{eq3**}
\frac{d \rho}{d \varphi} \, = \, \mathcal{F}(\varphi, \rho) := \frac{R(\varphi, \rho)}{\Theta(\varphi, \rho)}
\end{equation}
well defined in $C \backslash \Theta^{-1}(0)$ being the cylinder $C  \, = \, \left\{ (\theta, \rho) \in \mathbb{S}^1 \times \mathbb{R} \, : \, 0 \leq \rho \ll 1 \right\}$ with $\mathbb{S}^1 = \mathbb{R}/ (2 \pi \mathbb{Z})$. In Remark \ref{rem-choose-pq} we justify why we only consider in the polar blow-up weights in the set $W(\mathbf{N}(\mathcal{X}))$.

Let $F(x,y)=0$ be a real invariant analytic curve of $\mathcal{X}$ which always exists under the assumptions of Theorem \ref{teo-F-exists}.  Denote by $K(x,y)$ the {\it cofactor} of $F$, that is, $\mathcal{X}(F) = K F$ holds. In weighted polar coordinates this equation is transformed into $\hat{\mathcal{X}}(\hat{F}) = \hat{K} \hat{F}$ where $\hat{\mathcal{X}} = \partial_\varphi + \mathcal{F}(\varphi, \rho) \partial_\rho$, $\hat{F} = F \circ \phi$ and $\hat{K}$ is the cofactor of the invariant curve $\hat{F}=0$ of $\hat{\mathcal{X}}$, see equation \eqref{def-Khat} for its explicit expression.

Let $\rho(\varphi; \rho_0)$ be the solution of the Cauchy problem (\ref{eq3**}) with initial condition $\rho(0; \rho_0) = \rho_0 > 0$  sufficiently small, and $\hat\gamma_{\rho_0} = \{ (\varphi, \rho(\varphi; \rho_0) : 0 \leq \varphi \leq 2 \pi \} \subset C$ an arc of orbit of \eqref{eq3**}. We define
\begin{equation}\label{def-int-K}
\int_{\hat\gamma_{\rho_0}} \hat K = PV \int_{0}^{2 \pi} \hat K(\varphi, \rho(\varphi; \rho_0)) \, d \varphi,
\end{equation}
where $PV$ stands for the {\it Cauchy principal value}. We recall here that given a continuous function $f$ defined in $I \subset [0, 2 \pi] \backslash \Omega$ with $\Omega = \{ \theta^*_1, \ldots, \theta^*_\ell \}$, the Cauchy principal value of the integral $\int_I f(\theta) \, d \theta$ is defined as
$$
PV \int_I f(\theta) \, d \theta = \lim_{\varepsilon \to 0^+} \int_{I_\varepsilon} f(\theta) \, d \theta,
$$
when the limit exists. Here we have used the notation $I_\varepsilon = I \backslash J_\varepsilon$ with $J_\varepsilon = \cup_{i=1}^\ell (\theta_i^*-\varepsilon, \theta_i^*+\varepsilon)$. We emphasize that if an improper integral is convergent then it must converge to its Cauchy principal value. Of course $PV \int_I f(\theta) \, d \theta$ may exist though the improper integral $\int_I f(\theta) \, d \theta$ does not.
\newline

Now we are in position to state our second main result which is proved in section \ref{Sec-PT2}. It shows the existence of \eqref{def-int-K} as well as a characterization of the centers in terms of \eqref{def-int-K}.

\begin{theorem}\label{Th-centers-pq-period}
Let $\mathcal{X}$ be a family of analytic planar vector fields having a monodromic singular point at the origin and $K$ the cofactor associated to a real analytic invariant curve through the origin. Then $\int_{\hat\gamma_{\rho_0}} \hat K$ exists and moreover the origin is a center if and only if
$$
\int_{\hat\gamma_{\rho_0}} \hat K \equiv 0
$$
for any initial condition $\rho_0 > 0$ sufficiently small.
\end{theorem}

\begin{remark}
{\rm Theorem \ref{Th-centers-pq-period} becomes trivial in the particular case where the invariant curve $F=0$ is constructed with a first integral $F$ of $\mathcal{X}$, hence $K \equiv 0$ and so $\hat K \equiv 0$. In other words, if there is an analytic first integral in a neighborhood of a monodromic sigular point then it is a center.}
\end{remark}

\begin{remark}
{\rm  It is worth to emphasize that, in general, the function $\int_{\hat\gamma_{\rho_0}} \hat K$ cannot be extended by continuity at $\rho_0=0$. This statement is proved in Remark \ref{re-no-exteder-r0}. }
\end{remark}

It is clear that in the application of Theorem \ref{Th-centers-pq-period} to specific families $\mathcal{X}$ we need the flow $\rho(\varphi; \rho_0)$ of (\ref{eq3**}). Sometimes we can overcome this difficulty as it is shown in the following Corollary \ref{Corol-semidef} where we show a sufficient focus condition.

\begin{corollary} \label{Corol-semidef}
Let $\mathcal{X}$ be a family of analytic planar vector fields having a monodromic singular point at the origin. Let $K$ be the cofactor associated to a real analytic invariant curve passing through the origin. If there is a neighborhood $\mathcal{U} \subset \mathbb{R}^2$ of the origin where $K$ is sign-defined then the origin is a focus of $\mathcal{X}$.
\end{corollary}

\begin{remark}
{\rm A sufficient condition for $K$ to be sign-defined in $\mathcal{U}$ is that there exists a finite jet $\mathcal{J}^d K$ (the Taylor polynomial of $K$ at the origin of degree $d$) that is positive or negative defined in $\mathcal{U} \backslash \{(0,0)\}$.  Recall that the computation of $\mathcal{J}^d K$ is algorithmic as we will see in this work. }
\end{remark}

\begin{remark}
{\rm The reader can check how Corollary \ref{Corol-semidef} works in the trivial example given by the linear family of vector fields $\mathcal{X} = (-y+\lambda x) \partial_x + (x+\lambda y) \partial_y$ having a center at the origin when $\lambda =0$ and a focus otherwise by using its invariant curve $F(x,y) = x^2+y^2 = 0$ whose cofactor $K(x,y)= 2 \lambda$ is constant.}
\end{remark}

\begin{remark}\label{multiples-F}
{\rm If $\mathcal{X}$ has $\ell > 1$ invariant analytic curves $F_i=0$ with cofactors $K_i$ for $i=1, \ldots, \ell$, then it also has the $\ell$-parameter invariant analytic curve $F = \prod_i F_i^{m_i} = 0$ with arbitrary multiplicities $m_i \in \mathbb{N}$ and cofactor $K = \sum_i m_i K_i$. Sometimes we may found $m_i$ such that we can apply Corollary \ref{Corol-semidef} to $K$.  }
\end{remark}

\begin{remark}\label{V-Xr}
{\rm In order to compute $K_{\bar{r}}$ we could apply the theory of invariant branches developed in sections \ref{Sec-IB1} and \ref{Sec-IB2}. Anyway we can overcome some difficulties appearing in that theory, see Remark \ref{Demina-remark},  just by noticing that if $\mathcal{X} = \mathcal{X}_r + \cdots$, $F(x,y) = F_{s}(x,y) + \cdots$ and $K(x,y) = K_{\bar{r}}(x,y) +  \cdots$ are the $(p,q)$-quasihomogeneous expansions of the vector field $\mathcal{X}$ with analytic invariant curve $F = 0$ having cofactor $K$ then $F_s = 0$ is an invariant algebraic curve of $\mathcal{X}_r = P_{p+r}(x,y) \partial_x + Q_{q+r}(x,y) \partial_y$ with cofactor $K_{\bar{r}}$. Moreover, following the results of Algaba {\it et al.} \cite{AGR3} we know that the irreducible factors of $F_s$ are factors of the inverse integrating factor $V(x,y) = (p x, q y) \wedge \mathcal{X}_{r} = p x Q_{q+r}(x,y) - q y P_{p+r}(x,y)$ of $\mathcal{X}_r$. }
\end{remark}

Finally we focus on analytic families of vector fields $\mathcal{X}$ in the monodromic class ${\rm Mo}^{(p,q)}$. In this case $\Pi$ is analytic, hence can be expressed as $\Pi(\rho_0; \lambda) = \sum_{i \geq 1} \eta_{i}(\lambda) \rho_0^{i}$, where the generalized Poincar\'e-Lyapunov quantities $\eta_{i}$ are analytic functions of the parameters $\lambda$ of the family. The ring $\mathfrak{R}$ of germs of analytic functions at a point $\lambda^* \in \Lambda$ is noetherian. Hence the \emph{Bautin ideal} $\mathcal{B}$ in $\mathfrak{R}$ of the family $\mathcal{X}$ generated by all the $\eta_i$, which we denote by $\mathcal{B} = \langle \eta_1-1, \eta_2, \eta_3, \ldots \rangle$, is finitely generated.

In the next result we will use the following notation: given an ideal $\mathcal{I}$ in the ring $\mathfrak{R}$ we define $\tilde{\eta}_j = \eta_j \mod \mathcal{I}$ to denote that $\tilde{\eta}_j$ is  the remainder of $\eta_j$ upon division by a basis of $\mathcal{I}$.

Our last result shows the structure of the Poincar\'e map $\Pi$ in the monodromic class ${\rm Mo}^{(p,q)}$ and it is proved in subsection \ref{subsec-mopq}.

\begin{theorem}\label{Th-centers-Mopq}
Let $\mathcal{X} = \mathcal{X}_r + \cdots$ with $r \geq 1$ be the $(p,q)$-quasihomogeneous expansion of an analytic family of planar vector field  $\mathcal{X} \in {\rm Mo}^{(p,q)}$ having an analytic invariant curve $F(x,y) = 0$ passing through the origin whose $(p,q)$-quasihomogeneous expansion is given by $F(x,y) = F_{s}(x,y) + \cdots$. Then $K(x,y) = K_{\bar{r}}(x,y) +  \cdots$ with $\bar{r} \geq r$ is the $(p,q)$-quasihomogeneous expansion of the cofactor $K$ of $F$. Moreover,
\begin{equation}\label{Teo-int-K-alphas}
\int_{\hat\gamma_{\rho_0}} \hat K  = \sum_{i \geq \bar{r}-r} \beta_i \, \rho_0^i
\end{equation}
is not a singular integral and the analytic Poincar\'e map $\Pi$ has the Taylor expansion $\Pi(\rho_0) = \sum_{i \geq 1} \eta_{i} \rho_0^{i}$ where the following relations hold:
\begin{itemize}
  \item[(i)] If $\bar{r}-r = 0$ then $\beta_0 = \log(\eta_1^s)$ and $\beta_i = s \eta_{i+1} \mod \langle \eta_1-1, \eta_2, \ldots, \eta_i \rangle$ for $i > 0$.
  \item[(ii)] If $\bar{r}-r \geq 1$ then $\eta_1 =1$, $\eta_2 = \cdots = \eta_{\bar{r}-r} = 0$, $\beta_{\bar{r}-r} = s \eta_{\bar{r}-r+1}$, and $\beta_i = s \eta_{i+1} \mod \langle \eta_{\bar{r}-r+1}, \ldots, \eta_i \rangle$ for $i > \bar{r}-r$.
\end{itemize}
\end{theorem}

\section{Real analytic invariant curves} \label{SS3}

As a toy example illustrating some of the ideas in the proof of Theorem \ref{teo-F-exists} we can think in the linear family of vector fields $\mathcal{X} = (-y+\lambda x) \partial_x + (x+\lambda y) \partial_y$. Its complexification $\mathcal{X}^{\mathbb{C}} = (-y+\lambda x) \partial_x + (x+\lambda y) \partial_y$ has the separatrices $f(x,y) = x + i y = 0$ and $\overline{f(x,y)} = x - i y = 0$ (where the bar denotes taking complex conjugates of the coefficients) with cofactors $\kappa(x,y) = i + \lambda$ and $\overline{\kappa(x,y)} = -i + \lambda$, respectively, giving rise to the real analytic invariant curve $F^{\mathbb{R}}(x,y) = f(x,y) \overline{f(x,y)} = x^2+y^2 = 0$ with cofactor $K(x,y)= \kappa(x,y) + \overline{\kappa(x,y)} = 2 \lambda$.
\newline

Let $\mathbb{R} \{ x, y \}$ and $\mathbb{C} \{ x, y \}$ be the ring of real analytic functions in a neighborhood of the origin of $\mathbb{R}^2$ and the ring of holomorphic functions in a neighborhood of the origin of $\mathbb{C}^2$, respectively. We use the notation $\mathbb{K} \in \{ \mathbb{R}, \mathbb{C} \}$. Given $F \in \mathbb{K} \{x, y\}$ with $F(0,0) = 0$, we say that $F$ is reducible in $\mathbb{K}\{x, y\}$ if $F = F_1 F_2$ with $F_1, F_2 \in \mathbb{K}\{x, y\}$ and $F_1(0,0) = F_2(0,0) = 0$. Otherwise, we say that $F$ is {\it irreducible} in $\mathbb{K}\{x, y\}$. Any $F \in \mathbb{K}\{x, y\}$ can be decomposed in factors as $F(x,y) = U(x,y) \prod_{j=1}^k F_j^{m_j}(x,y)$ where the multiplicities $m_j$ are positive integers, $U, F_j \in \mathbb{K}\{x, y\}$, the $F_j$ are irreducible in $\mathbb{K}\{x, y\}$ and $U$ is a unit, that is, $U(0,0) \neq 0$. We say that $F, G \in \mathbb{K}\{x, y\}$ are coprime if they have no common non-unit irreducible factor in its decomposition over $\mathbb{K}\{x, y\}$.

\subsection{Proof of Theorem \ref{teo-F-exists}} \label{S-PT1}

\begin{proof}
By definition of monodromy, the origin must be an isolated singularity of the vector field $\mathcal{X} = P(x,y) \partial_x + Q(x,y) \partial_y$ in $\mathbb{R}^2$. But this does not implies that $P$ and $Q$ be coprime in the ring $\mathbb{R} \{ x, y \}$. It may occurs that $P$ and $Q$ have a common factor $\xi \in \mathbb{R} \{ x, y \}$ which is either a unit or it has an isolated zero at the origin. However the rescaled vector field $\mathcal{X}/\xi$ still has a monodromic singularity at the origin and now its components are real coprime. So we only study the case with $P$ and $Q$ coprime in $\mathbb{R} \{ x, y \}$.

Our strategy starts by taking the canonical complexification $\mathcal{X}^\mathbb{C}$ at $(\mathbb{C}^2, 0)$ of the real analytic vector field $\mathcal{X}$ at $(\mathbb{R}^2, 0)$. Thus we consider the complex extensions of $P$ and $Q$ (without changing its name) which are holomorphic in a neighborhood of the origin of $\mathbb{C}^2$ and $P(0,0) = Q(0,0) = 0$. There are two possibilities, either the singularity at the origin of $\mathcal{X}^\mathbb{C}$ in $\mathbb{C}^2$ is isolated or not.

The simplest case is when the origin is a non-isolated singular point of $\mathcal{X}^\mathbb{C}$. There it exists a common irreducible component $\xi \in \mathbb{C} \{ x, y \} \backslash \mathbb{R} \{ x, y \}$ of $P$ and $Q$ with $\xi(0,0)=0$. Since $\mathcal{X}$ is real it follows that the $\overline{\xi(x,y)}$ must be also a common irreducible component of $P$ and $Q$ where the bar denotes taking complex conjugates of the coefficients of $\xi$. Consequently $P$ and $Q$ must contain the common factor $\xi(x,y) \overline{\xi(x,y)}$ that turns into a real factor of $\mathcal{X}$ in contradiction with the coprimality hypothesis.
\newline

From now we assume that $\mathcal{X}^\mathbb{C}$ has an isolated singularity in $\mathbb{C}^2$ at the origin. We recall that {\it separatrix} for a germ of holomorphic vector field $\mathcal{Z}$ at $(\mathbb{C}^2, 0)$ is an invariant irreducible analytic curve passing through the origin. In the celebrated paper \cite{CS}, C. Camacho and P. Sad proved that if the origin in $\mathbb{C}^{2}$ is an isolated singularity of $\mathcal{Z}$ then there always exists at least one separatrix. Our strategy is to invoke the Camacho-Sad separatrix theorem to $\mathcal{X}^\mathbb{C}$. The outcome is the existence of a function $f \in \mathbb{C}\{x,y\}$ such that $f(0,0)=0$ and the curve $f(x,y)=0$ is invariant for the local flow of $\mathcal{X}^\mathbb{C}$. In particular, there is a complex cofactor $\kappa \in \mathbb{C}\{x,y\}$ such that $\mathcal{X}^\mathbb{C}(f) = \kappa f$. In consequence $\mathcal{X}(f) = \kappa f$, and taking complex conjugation (denoted by a bar) in the above equation gives $\mathcal{X}(\bar{f}) = \bar{\kappa} \bar{f}$. Therefore the real analytic function $F^{\mathbb{R}}(x,y) = f(x,y) \overline{f(x,y)}$ with $F^{\mathbb{R}}(0,0)=0$ is irreducible in $\mathbb{R}\{x,y\}$ and induces the real invariant analytic curve $F^{\mathbb{R}}(x,y)=0$ of $\mathcal{X}$ with real cofactor $K^{\mathbb{R}} = \kappa + \bar{\kappa}$.

The monodromy of $\mathcal{X}$ at the origin clearly implies that $F^{\mathbb{R}}$ must have an isolated zero in $\mathbb{R}^2$ at the origin.
\end{proof}

\subsection{Equivalence class of formal invariant curves} \label{subsec-equiv-class}

Let $F \in \mathbb{C}[[x,y]]$ be a formal power series with $F(0,0)=0$ such that $F(x,y)=0$ is a formal invariant curve of $\mathcal{X}$, that is, there is a formal cofactor $K \in \mathbb{C}[[x,y]]$ formally satisfying the equation $\mathcal{X}(F) = K F$. As it is emphasized in \cite{AGR3}, then $U F = 0$ is also a formal invariant curve of $\mathcal{X}$ for any formal unit $U$, that is, for any $U \in \mathbb{C}[[x,y]]$ with $U(0,0) \neq 0$. Thus an equivalence class is induced in the set of formal invariant curves of $\mathcal{X}$: two formal invariant curves $F_1=0$ and $F_2=0$ are equivalent (denoted by $F_1 \sim F_2$) if $F_1=F_2 U$ for some formal unit $U$. In \cite{AGR3} it is proved the relevant fact that, for each irreducible formal invariant curve at the origin, there always exists an irreducible analytic invariant curve such that both are equivalent. We add that, for any weights $(p,q) \in \mathbb{N}^2$, the $(p,q)$-quasihomogeneous expansion of two equivalent series $F_1$ and $F_2$ possess the same leading $(p,q)$-quasihomogeneous term.

It is worth to emphasize that, for any analytic unit $U$, there is always an analytic $K$ such that $\mathcal{X}(U) = K U$ holds just taking $K = \mathcal{X}(U)/U$. However it is obvious that the zero set $U=0$ can or cannot be an analytic invariant curve of $\mathcal{X}$. In the formal or analytic setting, the correct definition of {\it invariant unit curve} is that $\mathcal{X}(U) |_{U=0} \equiv 0$.

\subsection{Branches of analytic curves at singularities} \label{Sec-IB1}

In the following and along all the work we will only consider $F^{\mathbb{R}}$ defined in $\mathbb{R}^2$. Neither $x$ nor $y$ can divide $F^{\mathbb{R}}$ since the origin is an isolated real zero of $F^{\mathbb{R}}$. Therefore by Newton-Puiseux Theorem (see \cite{BK} for instance) there exists a finite factorization
\begin{equation}\label{factor-Puiseux}
F^{\mathbb{R}}(x,y) = u(x,y) \prod_{i} (y - y^*_i(x)) \prod_{i} (y - \bar{y}^*_i(x))
\end{equation}
where $u$ is a real analytic unit $u(0,0) \neq 0$, the $y^*_i(x)$ and $\bar{y}^*_i(x)$ are complex conjugated holomorphic functions of $x^{1/ n_i}$ with $y^*_i(0) = \bar{y}^*_i(0)= 0$ called {\it branches} of $F^{\mathbb{R}}$ at the origin, and the exponents $n_i$ are positive integers called the {\it indices} of the branches $y^*_i(x)$ and $\bar{y}^*_i(x)$. Recall that for a complex conjugated pair $(y_i^*(x), \bar{y}_i^*(x))$, we have that $(y - y_i^*(x))(y-\bar{y}_i^*(x)) = y^2- 2 y {\rm Re}(y_i^*(x)) + |y_i^*(x)|^2$ is real-valued in a half-neighborhood of $x=0$. In fact we can construct a real analytic invariant curve $F=0$ of $\mathcal{X}$ using a product $F(x,y) = \prod_{y_i^* \in \mathcal{S}} (y - y^*_i(x))$ where $\mathcal{S}$ is a convenient subset of the branches of $F^{\mathbb{R}}$. Indeed, for an irreducible $F$ the set $\mathcal{S}$ is just the conjugacy class $\mathbf{C}(y_1^*)$ of $y_1^*$, see \cite{C-A} and the proof of Proposition \ref{clever-prop}.
\newline

Using branching theory to an analytic function $F(x,y)$ with $F(0,0)=0$ we know that any branch $y_i^*$ of $F$ emerging from $(x, y) = (0,0)$ can be locally expressed as convergent Puiseux series determined by the descending sections of the Newton diagram of $F$. We recall that, given $F(x,y) = \sum_{i, j} f_{ij} x^i y^{j}$ with support ${\rm supp}(F) = \{ (j, i) \in \mathbb{N}^2 : f_{ij} \neq 0 \}$, the Newton diagram of $F$, denoted by $\mathbf{N}(F)$, is the boundary of the convex hull of the set $\bigcup_{(j,i) \in {\rm supp}(F} \{ (j,i) + \mathbb{R}^2_+\}$ where $\mathbb{R}^2_+$ is the positive quadrant. When we look at the local zero-set of $F$ near $(x,y)=(0,0)$ we will be only interested in the different segments of that polygon with rational negative slopes $-k_1/ k_2$, being $k_1$ and $k_2$ coprimes. Therefore one has that any branch $y_i^*(x)$ in \eqref{factor-Puiseux} has the form
\begin{equation}\label{branch-first}
y_i^*(x) = \alpha_0 x^{k_1/k_2} + o(x^{k_1/k_2})
\end{equation}
with $\alpha_0 \in \mathbb{C} \backslash \{ 0 \}$ and $0 < k_1/k_2 \in \mathbb{Q}$. Associated to each descending segment, the leading coefficients $\alpha_0$ are the nonzero roots of a {\it determining polynomial} $\mathcal{P}(\eta)$ associated to each descending segment of $\mathbf{N}(F)$, and defined by the first term of the expansion
\begin{equation}\label{defining-eq}
F(x^{k_2}, x^{k_1} \eta) = x^{m_1 k_1 + \hat{m}_1 k_2} \, \eta^{m_1} \big[ \mathcal{P}(\eta) + \mathcal{O}(x) \big],
\end{equation}
where the segment contains the points $(m_i, \hat{m}_i) \in \mathbb{N}^2$ with $i=1, \ldots, p$ ordered such that $m_1 < m_2 < \cdots < m_p$. By construction $\mathcal{P}$ is a polynomial having at most $m_p-m_1$ roots, the length of the horizontal projection of the segment.

Besides the leading term $\alpha_0 x^{k_1/k_2}$, the Newton-Puiseux algorithm is used to determine the higher order terms of the branch \eqref{branch-first} giving rise to the following convergent Puiseux series
\begin{equation}\label{branches}
y_i^*(x) = \sum_{j \geq 0} \alpha_j x^{\frac{k_1}{k_2} + \frac{j}{n_i}},
\end{equation}
for certain $n_i \in \mathbb{Z}^+$ called the {\it index} of the branch $y_i^*$. A particularly easy (and generic) case corresponds to the expansion \eqref{branch-first} of a {\it simple branch}, that is, a branch whose leading coefficient $\alpha_0$ is a simple root of its determining polynomial $\mathcal{P}$. The Puiseux series \eqref{branch-first} of a simple branch is \eqref{branches} with $n_i = k_2$, see a proof in \cite{VT}. So inserting \eqref{branches} with an initial string of undetermined coefficients $\alpha_j$ with $j \geq 1$ into condition $F(x, y_i^*(x)) \equiv 0$ gives the value of these uniquely determined coefficients for a simple branch.

A key point is that all the branches $y_i^*(x)$ are analytic functions of $\sigma = x^{1/ n_i}$. Therefore any branch $y_i^*$ associated to a segment of slope $-k_1/k_2$ and index $n_i$ can be expanded analytically near $\sigma=0$ as $y_i^*(\sigma^{n_i}) = \sigma^\frac{n_i k_1}{k_2} \sum_{j \geq 0} \alpha_j \sigma^{j}$ with $0 < n_i k_1/k_2 \in \mathbb{N}$ because $k_2$ divides $n_i$, and it satisfies $F(\sigma^{n_i}, y_i^*(\sigma^{n_i})) \equiv 0$ for all $\sigma$ close to zero.

\subsection{Invariant branches of $\mathcal{X}$ at singular points} \label{Sec-IB2}

On the other hand we say that a Puiseux series of the form \eqref{branch-first} with $k = k_1/k_2 \in \mathbb{Q}$ is an {\it invariant branch} of $\mathcal{X}$ at the origin if it satisfies term by term the differential equation $P(x,y) dy - Q(x,y) dx = 0$, that is, $y_i^*(x)$ must satisfy
\begin{equation}\label{eq-inv-branch}
P(x,y_i^*(x)) \frac{d y_i^*}{dx}(x) - Q(x, y_i^*(x)) = 0,
\end{equation}
for all $x$ in a half-neighborhood of $x=0$. We define $\mathcal{B}$ as the set of all the invariant branches of $\mathcal{X}$ at the origin. Clearly any element of $\mathcal{B}$ will be non-real by the monodromy of $\mathcal{X}$.

Of course all the curves $y - y^*_i(x) = 0$ (and also $y - \bar{y}^*_i(x) = 0$) are invariant for the vector field $\mathcal{X}$, that is, there is a cofactor $\kappa_i$ such that $\mathcal{X}(y-y_i^*) = \kappa_i(x,y) (y-y_i^*)$, see \cite{GGG}. This implies that all the branches $y^*_i(x)$ (and also $\bar{y}^*_i(x)$) of $F^{\mathbb{R}}$  are indeed invariant branches of $\mathcal{X}$.

The lower order terms in \eqref{eq-inv-branch} are the lower order terms of
\begin{equation}\label{eq-inv-branch-2}
\alpha_0 k x^{k-1} P(x,\alpha_0 x^{k}) - Q(x, \alpha_0 x^{k}) = 0,
\end{equation}
and in order that $\alpha_0 \neq 0$ there must be involved at least two monomials. To determine the allowed rational leading exponents $k$ and leading coefficient $\alpha_0$ of the possible invariant branches of $\mathcal{X}$ we can proceed either by simple inspection in \eqref{eq-inv-branch-2} or by two different (but equivalently) ways that we shortly explain. The first approach consists in using the Newton diagram $\mathbf{N}(\mathcal{X})$ of an analytic vector field $\mathcal{X}$, see section \ref{S1}. It is easy to see that selecting the weights $(p,q) \in W(\mathbf{N}(\mathcal{X}))$ one gets an expansion
\begin{equation}\label{campo-X}
\mathcal{X} = \sum_{j \geq r} \mathcal{X}_j
\end{equation}
where $r \geq 1$ and $\mathcal{X}_j = P_{p+j}(x,y) \partial_x + Q_{q+j}(x,y) \partial_y$ are $(p,q)$-quasihomogeneous vector fields of degree $j$. We call $\mathcal{X}_r$ the leading part of $\mathcal{X}$, which clearly depends on the chosen weights $(p,q)$.

It can be checked that, choosing $(p,q) \in W(\mathbf{N}(\mathcal{X}))$ and therefore having an expansion \eqref{campo-X}, the leading term $\alpha_0 x^{k_1/k_2}$ of any invariant branch \eqref{branch-first} of a vector field $\mathcal{X}$ must be a solution of the differential equation $P_{p+r}(x,y) dy - Q_{q+r}(x, y) dx = 0$ with $k_1/k_2 = q/p$ and $\alpha_0$ a root of a polynomial $\mathcal{Q}(\eta)$ defined by
\begin{equation}\label{det-pol-2}
P_{p+r}(x, \eta x^{q/p}) \frac{d}{d x}(\eta x^{q/p}) - Q_{q+r}(x, \eta x^{q/p}) = \mathcal{Q}(\eta) x^\frac{r+q}{p}.
\end{equation}
Condition \eqref{det-pol-2} implies that $y^p-\alpha_0 x^q = 0$ is an irreducible invariant algebraic curve of the leading vector field $\mathcal{X}_r$.

\begin{proposition}
Let $y_i^*(x) = \alpha_0 x^{q/p} + \cdots$ with $(p,q) \in W(\mathbf{N}(\mathcal{X}))$ be an invariant branch of $\mathcal{X}$, hence $\alpha_0$ is a root of the polynomial $\mathcal{Q}$ defined in \eqref{det-pol-2}. Assume that $y_i^*$ is also a branch of an analytic invariant curve $F=0$ of $\mathcal{X}$.  Moreover, let $F(x,y) = F_{s}(x,y) + \cdots$ be the $(p,q)$-quasihomogeneous expansion of $F$ where $F_s(x,y) \not\equiv 0$ is a $(p,q)$-quasihomogeneous polynomial of degree $s$ and the dots are higher order $(p,q)$-quasihomogeneous terms. Then $\alpha_0$ is also a root of the determining polynomial $\mathcal{P}$ defined by \eqref{defining-eq} with $k_1/k_2 = q/p$, which can be also computed as $F_s(1, \eta)= \eta^{m_1} \mathcal{P}(\eta)$.
\end{proposition}

\begin{proof}
Since $y_i^*$ is a branch of $F$ it follows that its leading term $\alpha_0 x^{q/p}$ must be a branch of $F_s$, that is, $F_s(x, \alpha_0 x^{q/p}) \equiv 0$. Since $F_s(x^p, \eta x^{q}) = x^{s} F_s(1, \eta)$ by the $(p,q)$-quasihomogeneity of $F_s$, we get $F_s(x, \eta x^{q/p}) = x^{s/p} F_s(1, \eta)$ and therefore we see that $\alpha_0$ must be a root of the polynomial $F_s(1, \eta)$. Indeed, comparing \eqref{defining-eq} where $k_1/k_2 = q/p$ with $F(x^p, \eta x^q) = x^s (F_s(1, \eta) + O(x))$ we see that $s=m_1 q+ \hat{m}_1 p$ and $F_s(1, \eta)= \eta^{m_1} \mathcal{P}(\eta)$.
\end{proof}

\begin{proposition}\label{clever-prop}
Let $\mathcal{X}$ be real analytic planar vector field with coprime components and a monodromic singularity at the origin. We fix the weights $(p,q) \in W(\mathbf{N}(\mathcal{X}))$ that defines the polynomial $\mathcal{Q}$ of degree $d$ by \eqref{det-pol-2}. Then there is a subset $\mathcal{R} = \{ \alpha_0^{(1)}, \alpha_0^{(2)}, \ldots, \alpha_0^{(\zeta)} \}$ of all roots of $\mathcal{Q}$ that are leading coefficients of invariant branches $y_j^*(x) = \alpha_0^{(j)} x^{q/p} + \cdots$ of $\mathcal{X}$ for $j=1, \ldots, \zeta$. Then $F(x,y) = \prod_{j=1}^\zeta  (y-y_j^*(x)) = 0$ is an invariant curve of $\mathcal{X}$ and $F$ has an analytic factor. Moreover, if $F$ itself were analytic then the associated determining polynomial $\mathcal{P}$ of $F$ is just $\mathcal{Q}$ provided that all the roots of $\mathcal{Q}$ are in $\mathcal{R}$.
\end{proposition}

\begin{proof}
By the analyticity and coprimality of $\mathcal{X}$ and the monodromy of the origin we can apply Theorem \ref{teo-F-exists} and conclude that there is an irreducible analytic invariant curve $f_1(x,y)=0$ of $\mathcal{X}$. Therefore all the branches of $f_1$ are invariant branches of $\mathcal{X}$ and the set $\mathcal{R}$ is non-empty. We pick up a branch of $f_1$, say $y_1^*(x)$, whose Puiseux expansion is $y_1^*(x) = x^{q/p} \sum_{j \geq 0} \alpha_j^{(1)} x^{j/n_1}$ with positive index $n_1 \in \mathbb{N}$ and $0 < q/p = \gamma/n$ and $\gamma \in \mathbb{N}$. Therefore we can write $y_1^*(x) = \sum_{j \geq \gamma} \alpha_{j-\gamma}^{(1)} \, x^{j/n_1}$. For each $n_1$-th root of the unity $\varepsilon_k$ (thus $\varepsilon_k^{n_1} = 1$ with $k=0, \ldots, n_1-1$) we define the conjugated series of $y_1^*$ as $\sigma_{\varepsilon_k}(y_1^*)(x) = \sum_{j \geq \gamma} \varepsilon_k^j \, \alpha_{j-\gamma}^{(1)} \, x^{j/n_1}$. The set of all conjugates of $y_1^*$ will be called the conjugacy class $\mathbf{C}(y_1^*)$ of $y_1^*$, that is,
$$
\mathbf{C}(y_1^*) = \{ y_1^* = \sigma_{\varepsilon_0}(y_1^*), \sigma_{\varepsilon_1}(y_1^*), \ldots, \sigma_{\varepsilon_{n_1-1}}(y_1^*) \}
$$
and its cardinality is $\# \mathbf{C}(y_1^*) = n_1$. We emphasize that it may happen that $y_1^*$ remains invariant under the map $\sigma_{\varepsilon_k}$ for some $k$.

By Lemma 1.2.3 of \cite{C-A} we know that if $y_1^*(x)$ is a branch of the analytic function $f(x,y)$ then all conjugates of $y_1^*(x)$ are branches of $f$ too. Thus $\mathbf{C}(y_1^*)$ is a set formed by branches of $f(x,y)$ having the same index $n_1$. Moreover, by Lemma 1.2.5 of \cite{C-A} we get
$$
f_1(x,y) = \prod_{k=0}^{n_1-1} (y-\sigma_{\varepsilon_k}(y_1^*)(x)).
$$
By construction $f_1$ is an analytic irreducible factor of $F$ and this proves the first part of the proposition.
\newline

Notice that $F$ would be analytic only in case that there is a partition of $\mathcal{R}$ where each class corresponds to leading coefficients of branches belonging to the same conjugacy class. In this case $F$ has  associated a determining polynomial $\mathcal{P}$ whose roots are the leading terms of all the branches of $F$, that is, $\alpha_0^{(j)}$ with $j=1, \ldots, \zeta$ are roots of $\mathcal{P}$. Thus if all the roots of $\mathcal{Q}$ are in $\mathcal{R}$ then $\mathcal{P}=\mathcal{Q}$.
\end{proof}

\begin{proposition}
Let $y_i^*(x) = \alpha_0 x^{q/p} + \cdots$ with $(p,q) \in W(\mathbf{N}(\mathcal{X}))$ be an invariant branch of $\mathcal{X}$. If $y_i^*$ comes from a irreducible invariant analytic curve $F(x,y)=0$ then the index $n_i$ of $y_i^*$ is bounded by $p \leq n_i \leq d = \deg(\mathcal{Q})$. Moreover, if $d$ and $p$ are even then $n_i$ is also even and $d$ divides $n_i$.
\end{proposition}

\begin{proof}
We know that $n_i = \lambda p$ with $\lambda \in \mathbb{N}$. All the invariant branches of $F$ belong to the conjugacy class $\mathbf{C}(y_i^*)$, hence all of them share the same leading exponent $q/p$ and index $n_i$. In particular all the elements in $\mathbf{C}(y_i^*)$ are invariant branches of $\mathcal{X}$ associated to the same segment of $\mathbf{N}(\mathcal{X})$ and consequently its associated leading coefficients are roots of $\mathcal{Q}$. This implies that the cardinality $\# \mathbf{C}(y_i^*)$ of $\mathbf{C}(y_i^*)$ is bounded by $d$ and we get
\begin{equation}\label{bound-ni}
p \leq n_i = \# \mathbf{C}(y_i^*) \leq d.
\end{equation}
Clearly when $d$ and $p$ are even then $n_i$ is also even and, by \eqref {bound-ni} $d$ divides $n_i$.
\end{proof}

\begin{remark} \label{Demina-remark}
{\rm There is a general method, see \cite{Br,Bru,De}, to know which is the index $n$ of an invariant branch $y^*$ and here we only sketch a part of it using our notation. To each $(p,q) \in W(\mathbf{N}(\mathcal{X}))$ we have the $(p,q)$-quasihomogeneous expansion $\mathcal{X} = \mathcal{X}_r + \cdots$ given in \eqref{campo-X} with
\[
\mathcal{X}_r = P_{p+r}(x,y) \partial_x + Q_{q+r}(x,y) \partial_y
\]
and we associate the dominant balance $E_0[y(x), x] = P_{p+r}(x,y) y' - Q_{p+r}(x,y)$ where the prime indicates derivative with respect to $x$. This dominant balance satisfies $E_0[\alpha_0 x^{q/p}, x] \equiv  0$ where $y^*(x)= \alpha_0 x^{q/p} + \cdots$. Now we calculate the formal G\^{a}teaux derivative of the dominant balance at $y(x) = \alpha_0 x^{q/p}$, that is,
\begin{eqnarray*}
\frac{\delta E_0}{\delta y}[\alpha_0 x^{q/p}] &=& \lim_{s \to 0} \frac{E_0[\alpha_0 x^{q/p} + s x^{q/p+j},x] - E_0[\alpha_0 x^{q/p}, x]}{s} = V(j) x^{\beta},
\end{eqnarray*}
for certain exponent $\beta(j, p, q)$. The zeros of $V(j)$ are called the {\it Fuchs indices} of the dominant balance. In the particular case that there are no Fuchs indices in $\mathbb{Q}^+ \backslash \mathbb{N}$ (non-natural positive rationals numbers) then the branch $y^*$ has index $n = p$. }
\end{remark}

\begin{remark} \label{simple+rotation-remark}
{\rm We also want to emphasize that there are other simpler ways than the explained in Remark \ref{Demina-remark} to conclude that $n=p$. The first method is based in Proposition \ref{clever-prop} and works when $\mathcal{P} = \mathcal{Q}$ because in that case the branch whose leading coefficient is a simple root of $\mathcal{Q}$ is just a simple branch and therefore $n=p$. A second method is based on the fact that, given a root $\alpha_0$ of $\mathcal{Q}$, sometimes it is possible to detect if it is a leading coefficient of a branch $y^*$ that belongs to a specific conjugacy class $\mathbf{C}(y^*)$ whose elements have leading coefficients that are also roots of $\mathcal{Q}$. Once $\mathbf{C}(y^*)$ is known we can easily determine the index $n$ of all the branches in $\mathbf{C}(y^*)$. }
\end{remark}

\section{The weighted polar blow-up} \label{S3}

The {\it weighted polar} blow-up $(x,y) \mapsto (\rho, \varphi)$ given by
\begin{equation}\label{wpbu}
x= \rho^p \cos\varphi, \ \ y = \rho^q \sin\varphi,
\end{equation}
wit Jacobian $J(\varphi, \rho) = \rho^{p+q-1}(p \cos^2\varphi + q \sin^2\varphi)$, brings the $(p,q)$-quasi\-ho\-mo\-ge\-neous vector field $\mathcal{X}_j = P_{p+j}(x,y) \partial_x + Q_{q+j}(x,y) \partial_y$ of degree $j$ into the system
$$
\dot{\rho} = \frac{\rho^{j+1} F_j(\varphi)}{D(\varphi)}, \ \ \dot{\varphi} = \frac{\rho^{j} G_j(\varphi)}{D(\varphi)},
$$
where
\begin{eqnarray} \label{def-GjFjD}
F_j(\varphi) &=& P_{p+j}(\cos\varphi, \sin\varphi) \cos\varphi + Q_{q+j}(\cos\varphi, \sin\varphi) \sin\varphi, \nonumber \\
G_j(\varphi) &=& p \, Q_{q+j}(\cos\varphi, \sin\varphi) \cos\varphi -q \, P_{p+j}(\cos\varphi, \sin\varphi) \sin\varphi,  \\
D(\varphi) &=& p \cos^2\varphi + q \sin^2\varphi > 0 \nonumber.
\end{eqnarray}
When we transform $\mathcal{X}$ under the weighted polar blow-up \eqref{wpbu}, the selected weights $(p,q)$ must be such that the minimum degree $r$ in the expansion \eqref{campo-X} of $\mathcal{X}$ be $r \geq 0$ in order to obtain an analytic polar vector field at $\rho=0$. The former is guaranteed because we always choose $(p,q) \in W(\mathbf{N}(\mathcal{X}))$. Therefore, after removing the common factor $\rho^{r} / D(\varphi) > 0$, $\mathcal{X}$ is transformed into the form
\begin{equation}\label{system-w-polares}
\dot{\rho} = R(\varphi, \rho) = \sum_{j \geq r} \rho^{j-r+1} F_j(\varphi), \ \ \dot{\varphi} = \Theta(\varphi, \rho) = \sum_{j \geq r} \rho^{j-r} G_j(\varphi).
\end{equation}

\begin{remark} \label{rem-choose-pq}
{\rm We want to mention that there are other weights not belonging to $W(\mathbf{N}(\mathcal{X}))$ producing a leading degree $r \geq 0$ of $\mathcal{X}$, but we only consider weights lying in $W(\mathbf{N}(\mathcal{X}))$ because they are necessary and sufficient to study the monodromy of a singularity of $\mathcal{X}$ which is the only setting we are interested in. Moreover it is clear that for any invariant analytic curve $F(x,y)=0$ of $\mathcal{X}$ it must occur that the weights $(p,q)$ of $\mathbf{N}(F)$ must be also weights of $\mathbf{N}(\mathcal{X})$.
}
\end{remark}

We define the sufficiently small cylinder
$$
C  \, = \, \left\{ (\theta, \rho) \in \mathbb{S}^1 \times \mathbb{R} \, : \, 0 \leq \rho \ll 1 \right\} \ \mbox{with } \mathbb{S}^1 = \mathbb{R}/ (2 \pi \mathbb{Z}),
$$
and we consider the ordinary differential equation of the orbits of \eqref{system-w-polares}, that is,
\begin{equation}\label{eq3*}
\frac{d \rho}{d \varphi} \, = \, \mathcal{F}(\varphi, \rho) = \sum_{i \geq 1} \mathcal{F}_i(\varphi) \rho^i,
\end{equation}
where $\mathcal{F}_1(\varphi) = F_r(\varphi)/G_r(\varphi)$. Then $\mathcal{F}$ is a function well defined in $C \backslash \Theta^{-1}(0)$. We say that $\varphi = \varphi^*$ is a $(p,q)$--{\it characteristic direction} for the origin of the vector field $\mathcal{X}$ if $G_r(\varphi^*)=0$ and we define $\Omega_{pq}$ as the set of all the $(p,q)$--characteristic directions:
$$
\Omega_{pq} = \{ \varphi^* \in \mathbb{S}^1 : G_r(\varphi^*) = 0 \} = \{ \varphi^*_1, \ldots, \varphi^*_\ell \}.
$$
Since $\dot{\rho} = O(\rho)$, $\dot{\varphi} = G_r(\varphi) + O(\rho)$, we see that $\{\rho=0 \}$ is an invariant set of \eqref{system-w-polares} which is either a periodic orbit or a polycycle of \eqref{eq3*}.

\begin{definition}\label{Mo-pq}
Given some weights $(p,q) \in W(\mathbf{N}(\mathcal{X}))$, we say that the vector field \eqref{campo-X} belongs to the monodromic class ${\rm Mo}^{(p,q)}$ if its leading $(p,q)$-quasihomoge\-neous vector field $\mathcal{X}_r$ has a monodromic singularity at the origin.
\end{definition}

\begin{remark} \label{remark-Mopq}
{\rm Notice that $\mathcal{X} \in {\rm Mo}^{(p,q)}$ is equivalent to have $\Omega_{pq} = \emptyset$ and therefore $\{\rho=0\}$ is a periodic orbit whose associated Poincar\'e map $\Pi$ is analytic. We also want to remark that if  $\mathcal{X} \in {\rm Mo}^{(p,q)}$ then $\mathbf{N}(\mathcal{X})$ cannot possess more than one edge. The reason is because if there is an interior vertex (a vertex of $\mathbf{N}(\mathcal{X})$ not lying on the axis) then the leading vector field $\mathcal{X}_r$ associated to each coincident edge has either $x=0$ or $y=0$ invariant line, see for example \cite{AGR}.}
\end{remark}

\section{Integral of the cofactor} \label{S6}

Let $F=0$ be an invariant curve of $\mathcal{X}$ with cofactor $K$. Assume that $\mathcal{Y}$ is orbitally equivalent to $\mathcal{X}$, that is, $\mathcal{Y} = \phi_*(\mu \mathcal{X})$ where $\phi_*$ denotes the pull-back of the transformation $\phi$. Then $F \circ \phi = 0$ an invariant curve of $\mathcal{Y}$ with cofactor $(K \mu) \circ \phi$. In the particular case that $\phi(\varphi, \rho) = (\rho^p \cos\varphi, \rho^q \sin\varphi)$ the equation $\mathcal{X}(F) = K F$ is transformed into $\hat{\mathcal{X}}(\hat{F}) = \hat{K} \hat{F}$ where $\hat{\mathcal{X}} = \partial_\varphi + \mathcal{F}(\varphi, \rho) \partial_\rho$, $\hat{F} = F \circ \phi$ and $\hat{K} = (K \circ \phi) D / (\rho^r \Theta)$, that is,
\begin{equation}\label{def-Khat}
\hat{K}(\varphi, \rho) = \frac{D(\varphi) K(\rho^p \cos\varphi, \rho^q \sin\varphi)}{\rho^r \Theta(\varphi, \rho)}.
\end{equation}

\subsection{Proof of Theorem \ref{Th-centers-pq-period}} \label{Sec-PT2}

\begin{proof}
Let $F(x,y) = 0$ be the analytic invariant curve of $\mathcal{X}$ with $F(0,0)=0$. Of course $(x,y)=(0,0)$ is an isolated real zero of $F(x,y)$ by the monodromy of $\mathcal{X}$ at the origin. Let $F(x,y) = F_{s}(x,y) + \cdots$ with $F_s(x,y) \not\equiv 0$ be the $(p,q)$-quasihomogeneous expansion of $F$ with $(p,q) \in W(\mathbf{N}(\mathcal{X}))$. Since $(x,y)=(0,0)$ is an isolated real zero of $F$, $s \geq 1$ and we consider the expression
\begin{equation}\label{F-polar-expression}
\hat{F}(\varphi, \rho) = F(\rho^p \cos\varphi, \rho^q \sin\varphi) = \rho^s \, [F_{s}(\cos\varphi, \sin\varphi) + o(\rho)].
\end{equation}
We claim that the eventual real roots of $F_s(\cos\varphi, \sin\varphi)$ must belong to $\Omega_{pq}$. The reason is that $\hat{F}(\varphi, \rho) / \rho^s = 0$ is an invariant curve of the differential equation \eqref{eq3*} that only can intersect the polycycle $\rho=0$ at its singularities and the intersection is isolated by monodromy.

Therefore $\hat{F}(\varphi, \rho)$ has an isolated zero at $\rho=0$ for any $\varphi \in \mathbb{S}^1$ and there is an interval $I = (0, m] \subset \mathbb{R}$ such that the restriction $\hat{F}|_{\mathbb{S}^1 \times I} \neq 0$. Moreover, $\hat{F}(\varphi, \rho)$ is a composition of analytic functions so analytic itself.

The flow $\rho(\varphi; \rho_0)$ is bounded for any $\rho_0 > 0$ sufficiently small and $\varphi \in  \mathbb{S}^1$ by monodromy again. Then there is a function $M(\rho_0)$ such that $0 < \rho(\varphi; \rho_0) < M(\rho_0)$. Clearly, by the uniqueness of the solutions of the Cauchy problem (\ref{eq3*}), taking a $\rho_0$ sufficiently small we obtain that $M(\rho_0) \in I$ and therefore $(\varphi, \rho(\varphi; \rho_0)) \in \mathbb{S}^1 \times I$. Consequently, the composition $\hat{F}(\varphi, \rho(\varphi; \rho_0))$ is also bounded and $\hat{F}(\varphi, \rho(\varphi; \rho_0)) \neq 0$ for any $\rho_0 > 0$ sufficiently small and $\varphi \in  \mathbb{S}^1$.
\newline

First we describe the simplest case assuming that $\Theta^{-1}(0) \backslash \{\rho = 0\} = \emptyset$. By the chain rule,
\begin{eqnarray*}
\frac{d}{d \varphi} \hat{F}(\varphi, \rho(\varphi; \rho_0)) &=& \frac{\partial \hat{F}}{\partial \varphi}(\varphi, \rho(\varphi; \rho_0)) + \frac{\partial \hat{F}}{\partial \rho}(\varphi, \rho(\varphi; \rho_0)) \frac{d \rho}{d \varphi}(\varphi; \rho_0) \\
 &=& \frac{\partial \hat{F}}{\partial \varphi}(\varphi, \rho(\varphi; \rho_0)) + \frac{\partial \hat{F}}{\partial \rho}(\varphi, \rho(\varphi; \rho_0)) \mathcal{F}(\varphi, \rho(\varphi; \rho_0))
\end{eqnarray*}
cannot diverge for any $\varphi \in \mathbb{S}^1$, and in particular
$$
\frac{\frac{d}{d \varphi} \hat{F}(\varphi, \rho(\varphi; \rho_0))}{\hat{F}(\varphi, \rho(\varphi; \rho_0))}
$$
is bounded for any $0 < \rho_0 \ll 1$ and $\varphi \in  \mathbb{S}^1$. The equation $\hat{\mathcal{X}}(\hat{F}) = \hat{K} \hat{F}$ with $\hat{\mathcal{X}} = \partial_\varphi + \mathcal{F}(\varphi, \rho) \partial_\rho$ gives
\begin{equation} \label{eq-der}
\frac{d}{d \varphi} \hat{F}(\varphi, \rho(\varphi; \rho_0)) = \hat{K}(\varphi, \rho(\varphi; \rho_0)) \hat{F}(\varphi, \rho(\varphi; \rho_0))
\end{equation}
and therefore
\begin{eqnarray*}
\int_{\hat\gamma_{\rho_0}} \hat K &=& \int_{0}^{2 \pi} \hat K(\varphi, \rho(\varphi; \rho_0)) \, d \varphi = \int_{0}^{2 \pi} \frac{\frac{d}{d \varphi} \hat{F}(\varphi, \rho(\varphi; \rho_0))}{\hat{F}(\varphi, \rho(\varphi; \rho_0))} \, d \varphi \\ &=& \bar{P}(2\pi; \rho_0) - \bar{P}(0; \rho_0),
\end{eqnarray*}
where
\begin{equation}\label{def-PP}
\bar{P}(\varphi; \rho_0) = \log |\hat{F}(\varphi, \rho(\varphi; \rho_0))|
\end{equation}
is a continuous primitive of the integrand on $[0, 2 \pi]$.
\newline

We continue assuming the worse situation in which $\Theta^{-1}(0) \backslash \{\rho = 0\} \neq \emptyset$. In particular,  for $\rho_0 \ll 1$, there exists $\bar{\varphi} \in \mathbb{S}^1$ such that
$$
(\bar\varphi, \rho(\bar\varphi; \rho_0)) \in \hat\gamma_{\rho_0} \cap \Theta^{-1}(0).
$$
Then the limit
$$
\lim_{\varphi \to \bar\varphi^-} \frac{d \rho}{d \varphi}(\varphi; \rho_0) = \lim_{\varphi \to \bar\varphi^-}  \mathcal{F}(\varphi, \rho(\varphi; \rho_0)) = \lim_{\varphi \to \bar\varphi^-}  \frac{R(\varphi, \rho(\varphi; \rho_0))}{\Theta(\varphi, \rho(\varphi; \rho_0))}
$$
may not exists. For this it is enough that $R(\bar\varphi, \rho(\bar\varphi; \rho_0)) \neq 0$.

Now from equation \eqref{eq-der} we get
\begin{eqnarray*}
\int_{\hat\gamma_{\rho_0}} \hat K &=& PV \int_{0}^{2 \pi} \hat K(\varphi, \rho(\varphi; \rho_0)) \, d \varphi = PV \int_{0}^{2 \pi} \frac{\frac{d}{d \varphi} \hat{F}(\varphi, \rho(\varphi; \rho_0))}{\hat{F}(\varphi, \rho(\varphi; \rho_0))} \, d \varphi.
\end{eqnarray*}
To compute this principal value we first introduce the set
$$
\bar\Omega_{pq}(\rho_0) = \{ \bar\varphi_1(\rho_0), \ldots, \bar\varphi_\kappa(\rho_0) \} \subset \mathbb{S}^1
$$
formed by all the angles $\bar\varphi_i(\rho_0)$ such that $(\bar\varphi_i(\rho_0), \rho(\bar\varphi_i(\rho_0); \rho_0)) \in \hat\gamma_{\rho_0} \cap \Theta^{-1}(0)$ for $i=1, \ldots, \kappa$. Now we define the set $\bar{I}_\varepsilon(\rho_0) = [0, 2 \pi] \backslash \bar{J}_\varepsilon(\rho_0)$ with $\bar{J}_\varepsilon(\rho_0) = \cup_{i=1}^\kappa (\bar\varphi_i(\rho_0)-\varepsilon, \bar\varphi_i(\rho_0)+\varepsilon)$ so that
\begin{eqnarray*}
\int_{\hat\gamma_{\rho_0}} \hat K  &=& \lim_{\varepsilon \to 0^+}  \int_{\bar{I}_\varepsilon(\rho_0)} \frac{\frac{d}{d \varphi} \hat{F}(\varphi, \rho(\varphi; \rho_0))}{\hat{F}(\varphi, \rho(\varphi; \rho_0))} \, d \varphi \\
 &=&  \lim_{\varepsilon \to 0^+} \sum_i \bar{P}(\bar\varphi_{i+1}(\rho_0) - \varepsilon; \rho_0) - \bar{P}(\bar\varphi_i(\rho_0) + \varepsilon; \rho_0) \\
 &=& \bar{P}(2\pi; \rho_0) - \bar{P}(0; \rho_0),
\end{eqnarray*}
where $\bar{P}(\varphi; \rho_0)$ was defined in \eqref{def-PP} and is a continuous primitive of the integrand on $\bar{I}_\varepsilon(\rho_0)$.
\newline

In summary, independently of whether $\Theta^{-1}(0) \backslash \{\rho = 0\}$ is empty or not,
\begin{eqnarray}\label{Int-cof*}
\int_{\hat\gamma_{\rho_0}} \hat K &=& \log |\hat F(2 \pi, \Pi(\rho_0))|  - \log |\hat F(0, \rho_0)| = \log \left| \frac{\hat F(0, \Pi(\rho_0))}{\hat F(0, \rho_0)} \right|,
\end{eqnarray}
where in the last step we have used that $\hat F$ is $2 \pi$-periodic in $\varphi$. Since $F(0,0)=0$ we claim that $\hat F(0, \rho)$ depends on $\rho$, hence clearly the former integral vanishes identically if and only if $\Pi(\rho_0) = \rho_0$, that is, the origin is a center finishing the proof. The former claim is true because, evaluating \eqref{F-polar-expression} at $\varphi=0$ we obtain that $\hat{F}(0, \rho) = \rho^s \, [F_{s}(1, 0) + o(\rho)]$ with $s \geq 1$ because $F(0,0)=0$. Here we have two possibilities, either $0 \not\in \Omega_{pq}$ so that $F_{s}(1, 0) \neq 0$ and the claim follows or $0 \in \Omega_{pq}$. This last option is removed arguing as follows: if $0 \in \Omega_{pq}$ then $\hat F(0, \rho)$ is independent on $\rho$ only when $\hat F(0, \rho) \equiv 0$ in which case it would exist an invariant ray $\{ \varphi =  0 \}$ which is forbidden by monodromy.

We end just by noticing that if $F(0,0) \neq 0$ then the previous argumentations can be false because $s=0$, and $\hat{F}(0, \rho) = F_{0} + o(\rho)$ with $F_{0} \neq 0$. It can occurs the phenomenon $\hat{F}(0, \rho) = F_{0}$, hence independent of $\rho$. For instance take $F(x,y)= 1+y$ to see that.
\end{proof}

\subsection{Proof of Corollary \ref{Corol-semidef}}  \label{Sec-cOR-SEMIDEF}

\begin{proof}[Proof of Corollary \ref{Corol-semidef}]
By definition \eqref{def-Khat} and \eqref{def-int-K} we have
\begin{equation}\label{eq-int-k}
\int_{\hat\gamma_{\rho_0}} \hat K = PV \int_{0}^{2 \pi} \frac{D(\varphi) K(\rho^p(\varphi; \rho_0) \cos\varphi, \rho^q(\varphi; \rho_0) \sin\varphi)}{\rho^r(\varphi; \rho_0) \Theta(\varphi, \rho(\varphi; \rho_0))}  \, d \varphi.
\end{equation}
Taking $\rho_0 > 0$ sufficiently small it follows that $\rho(\varphi; \rho_0) > 0$ and bounded for all $\varphi \in \mathbb{S}^1$. Moreover $\Theta(\varphi, \rho(\varphi; \rho_0)) \geq 0$ by the monodromy of $\mathcal{X}$ and clearly $D(\varphi) > 0$. In summary, if $K(x,y)$ is sign-defined in a neighborhood of the origin, then $\int_{\hat\gamma_{\rho_0}} \hat K \neq 0$ and the origin is a focus by Theorem \ref{Th-centers-pq-period}.
\end{proof}

\begin{remark}
{\rm If in the statement of Corollary \ref{Corol-semidef} we assume the stronger condition that $K$ is positive (or negative) defined in $\mathcal{U} \backslash \{(0,0)\}$ instead of only to be sign-defined in $\mathcal{U}$ then we have an alternative proof of it just by using a Lyapunov function argument as follows. Let $F=0$ be an analytic invariant curve of $\mathcal{X}$ predicted by Theorem \ref{teo-F-exists}. Since the curve has an isolated zero in $\mathbb{R}^2$ at the origin, there is a neighborhood $\mathcal{U}$ of the origin such that the function $F$ restricted to $\mathcal{U}$ is positive (or negative) defined in $\mathcal{U} \backslash \{(0,0)\}$. Since $K$ share that property by hypothesis, and taking into account that $\mathcal{X}(F) = K F$, we deduce that $F$ is indeed a Lyapunov function for $\mathcal{X}$ in $\mathcal{U}$ and consequently the origin is a focus. }
\end{remark}

\begin{remark}
{\rm We want to observe that the consequence of Corollary \ref{Corol-semidef} is not true in general if we only assume that the analytic cofactor $K(x,y)$ has a $(p, q)$-quasihomogeneous expansion $K(x,y) = \sum_{j \geq {\bar r}} K_j(x,y)$ with sign-defined leading term $K_{{\bar r}}$ because under this assumption, and given any continuous flow function $\rho: \mathbb{S}^1 \to (0, \varepsilon] \subset \mathbb{R}$, we cannot guarantee that
$$
\int_{0}^{2 \pi} K(\rho^p(\varphi; \rho_0) \cos\varphi, \rho^q(\varphi; \rho_0) \sin\varphi) \, d \varphi \neq 0
$$
for all $\varepsilon > 0$ small enough (equivalently for all $\rho_0>0$ small enough). The reason is because, although the function $K(\rho^p(\varphi; \rho_0) \cos\varphi, \rho^q(\varphi; \rho_0) \sin\varphi)$ adopts the form $\rho^d(\varphi; \rho_0) K_{\bar r}(\cos\varphi, \sin\varphi) + \rho^{d+1}(\varphi; \rho_0) R(\varphi, \rho(\varphi; \rho_0))$ with $K_{\bar r}(\cos\varphi, \sin\varphi)  \geq 0$ for all $\varphi \in \mathbb{S}^1$, the term $|R(\varphi, \rho(\varphi; \rho_0))|$ may not have a uniform upper bound in $\mathbb{S}^1$ independent of $\rho_0$. }
\end{remark}

\begin{remark}\label{re-no-exteder-r0}
{\rm  It is worth to emphasize that expression \eqref{eq-int-k} cannot be extended by continuity at $\rho_0=0$. If that extension were possible, using that $\rho(\varphi; 0) = 0$, then we would have
$$
\int_{\{ \rho_0=0 \}} \hat K = \left\{ \begin{array}{lll}
                                        0 & \mbox{if} & \bar{r} > r, \\
                                        \displaystyle PV \int_0^{2\pi} \frac{D(\varphi) \, K_{\bar{r}}(\cos\varphi, \sin\varphi)}{G_r(\varphi)} d \varphi & \mbox{if} & \bar{r} =r,
                                       \end{array}  \right.
$$
in case that this principal value exists. Moreover from the $(p,q)$-quasihomo\-geneous expansion $F(x,y) = F_{s}(x,y) +  \cdots$ we could use equation \eqref{Int-cof*} to express
\begin{eqnarray*}
\int_{\hat\gamma_{\rho_0}} \hat K &=& \log \left| \frac{\hat F(0, \Pi(\rho_0))}{\hat F(0, \rho_0)} \right| = \log \left| \frac{\Pi^s(\rho_0) F_s(\cos\varphi, \sin\varphi)) + O(\Pi^{s+1}(\rho_0))}{\rho_0^s (F_s(\cos\varphi, \sin\varphi) + O(\rho_0))} \right|
\end{eqnarray*}
whose extension to $\rho_0=0$ would give
$$
\int_{\{ \rho_0=0 \}} \hat K = \log(\eta_1^s)
$$
taking into account that $\Pi(\rho_0) = \eta_1 \rho_0 + o(\rho_0)$. Comparing both expressions of $\int_{\{ \rho_0=0 \}} \hat K$ we would have that if $\bar{r} > r$ then $\eta_1 = 1$ whereas
  $$
  \log\left( \eta_1^s \right) = PV \int_0^{2\pi} \frac{D(\varphi) \, K_{\bar{r}}(\cos\varphi, \sin\varphi)}{G_r(\varphi)} d \varphi,
  $$
  when $\bar{r} = r$ and this principal value exists. We can see that this expression of $\eta_1$ is wrong in several examples, see for instance Remark \ref{rem-NO-extension}. }
\end{remark}

\section{Vector fields in the monodromic class ${\rm Mo}^{(p,q)}$} \label{S7}

We recall that $\mathcal{X} \in {\rm Mo}^{(p,q)}$ if and only if $G_r(\varphi)$ has no real roots in $[0, 2 \pi)$. Therefore, in particular, $\Theta^{-1}(0) = \emptyset$ when $\mathcal{X} \in {\rm Mo}^{(p,q)}$.
Moreover, the origin is a center of $\mathcal{X}_r$ if and only if, additionally, $\int_0^{2 \pi} \mathcal{F}_1(\varphi) \, d \varphi = 0$, see \cite{GGL} for example. It can be checked  (using the Bautin method or just by the reciprocal of Theorem 5 in \cite{AFG}) that if the origin is a center of $\mathcal{X} \in {\rm Mo}^{(p,q)}$ then it is also a center of $\mathcal{X}_r$. In other words, $\int_0^{2 \pi} \mathcal{F}_1(\varphi) \, d \varphi = 0$ is a necessary center condition at the origin of $\mathcal{X} \in {\rm Mo}^{(p,q)}$.
\newline

For vector fields lying in ${\rm Mo}^{(p,q)}$, the solution $\rho(\varphi; \rho_0)$ of the Cauchy problem (\ref{eq3*}) with initial condition $\rho(0; \rho_0) = \rho_0$ admits the convergent Taylor development $\rho(\varphi; \rho_0) = \sum_{i \geq 1} a_i(\varphi) \rho_0^i$ with $a_1(0)=1$ and $a_i(0) = 0$ for all $i \geq 2$.

\begin{lemma}\label{lemma1-Mopq}
If $\mathcal{X} \in {\rm Mo}^{(p,q)}$ then the following holds.
\begin{itemize}
  \item[(i)] The function $a_1(\varphi)$ is expressed as
\begin{equation}\label{a1-eq}
a_1(\varphi) = \exp\left(\int_0^\varphi \frac{F_r(\theta)}{G_r(\theta)} d \theta \right),
\end{equation}
where $F_r$ and $G_r$ are defined in \eqref{def-GjFjD}.

  \item[(ii)] Given the $(p,q)$-quasihomogeneous expansion $F(x,y) = F_{s}(x,y) + \cdots$ of an analytic invariant curve $F(x,y) = 0$, one has $F_s(\cos\varphi, \sin\varphi) \neq 0$ for all $\varphi \in \mathbb{S}^1$.
\end{itemize}
\end{lemma}

\begin{proof}
Statement (i) is straightforward after imposing that $\rho(\varphi; \rho_0) = a_1(\varphi) \rho_0 + O(\rho_0^2)$ be a solution of (\ref{eq3*}). We reach that $a_1(\varphi)$ must satisfies the linear Cauchy problem
\begin{equation}\label{edo-a1}
G_r(\varphi) a'_1(\varphi) = a_1(\varphi) \, F_r(\varphi), \ \ a_1(0)=1.
\end{equation}

Statement (ii) is a straight consequence of the proof of Theorem \ref{Th-centers-pq-period} where we show that the eventual zeros of the function $F_s(\cos\varphi, \sin\varphi)$ lie in $\Omega_{pq}$.
\end{proof}

\subsection{Proof of Theorem \ref{Th-centers-Mopq}} \label{subsec-mopq}

\begin{proof}
The origin is a monodromic singular point of $\mathcal{X}$ and $\Pi$ is analytic at $\rho_0=0$ just because $\mathcal{X} \in {\rm Mo}^{(p,q)}$ and therefore $\Omega_{pq} = \emptyset$.

The fact that $\bar{r} \geq r$ easily follows by analyzing the lowest $(p,q)$-quasihomo\-geneous degrees in the equation $\mathcal{X}(F) = K F$ that gives $\mathcal{X}_r(F_s) = K_{r} F_s$ where $K_r$ can be eventually identically zero.
\newline

We start with the expression
\begin{equation}\label{hat-F}
\hat{F}(\varphi, \rho(\varphi; \rho_0))  = \rho_0^s \, [a_1^s(\varphi) F_{s}(\cos\varphi, \sin\varphi) + o(\rho_0)].
\end{equation}
From Lemma \ref{lemma1-Mopq} we know that $a_1(\varphi) \neq 0$ for all $\varphi \in [0, 2 \pi]$ and $F_s(\cos\varphi, \sin\varphi) \neq 0$ in all $\mathbb{S}^1$. Therefore, from the expression \eqref{hat-F} we check that $\hat{F}(\varphi, \rho(\varphi; \rho_0)) \neq 0$ for any $\rho_0 > 0$ sufficiently small and $\varphi \in [0, 2 \pi]$. This implies that $\int_{\hat\gamma_{\rho_0}} \hat K$ is not a singular integral for any $\rho_0 > 0$ sufficiently small. Indeed, following the lines of \eqref{Int-cof*} we have that
\begin{equation}\label{Int-cof}
\int_{\hat\gamma_{\rho_0}} \hat K = \int_{0}^{2 \pi} \frac{\frac{d}{d \varphi} \hat{F}(\varphi, \rho(\varphi; \rho_0))}{\hat{F}(\varphi, \rho(\varphi; \rho_0))} \, d \varphi = \log \left| \frac{\hat F(0, \Pi(\rho_0))}{\hat F(0, \rho_0)} \right|
\end{equation}
vanishes identically if and only if the origin is a center. More specifically, using that $\Pi(\rho_0) = \sum_{i \geq 1} \eta_i \rho_0^i$ and $\hat F(\varphi, \rho) = \sum_{i \geq s} F_i(\cos\varphi, \sin\varphi) \rho^i$ for certain $s \in \mathbb{N} \backslash \{0\}$ because $F(0,0)=0$. Taking $\varphi = 0$ we know that $F_s(1, 0) \neq 0$ and the Taylor expansion at $\rho_0=0$ of the right-hand part in \eqref{Int-cof} gives
\begin{equation}\label{log-expansion}
\int_{\hat\gamma_{\rho_0}} \hat K  = \log(\eta_1^s) + \sum_{i \geq 1} b_i \rho_0^i.
\end{equation}
Here the computations show that, if $\eta_1 = 1$ and $\eta_2 = \cdots = \eta_j = 0$ for some $j \geq 2$ then $b_1 = \cdots = b_{j-1} = 0$ and $b_j = s \eta_{j+1}$. In other words
\begin{equation}\label{relation-b-eta}
b_j = s \eta_{j+1} \mod \langle \eta_1-1, \eta_2, \ldots, \eta_j \rangle.
\end{equation}

On the other hand, the transformed cofactor $\hat{K}$ is
\begin{eqnarray*}
\hat{K}(\varphi, \rho) &=& \frac{D(\varphi) K(\rho^p \cos\varphi, \rho^q \sin\varphi)}{\rho^r \Theta(\varphi, \rho)} = \frac{D(\varphi) \sum_{i \geq \bar{r}} K_i(\cos\varphi, \sin\varphi) \rho^i}{\sum_{j \geq r} \rho^{j}  G_j(\varphi)} \\
 &=& \frac{D(\varphi) K_{\bar{r}}(\cos\varphi, \sin\varphi)}{G_r(\varphi)} \rho^{\bar{r}-r} +  O(\rho^{\bar{r}-r+1})
\end{eqnarray*}
where $G_r(\varphi) \neq 0$ for all $\varphi$ and therefore $\hat{K}(\varphi, \rho)$ is analytic at $\rho=0$.

Using that
\begin{eqnarray*}
\rho^{\bar r-r}(\varphi; \rho_0) &=& (a_1(\varphi) \rho_0 + a_2(\varphi) \rho_0^2 + O(\rho_0^3))^{\bar r-r} \\
 &=& a_1^{\bar r-r}(\varphi) \rho_0^{\bar r-r} + ({\bar r}-r) a_1^{\bar r-r-1}(\varphi) a_2(\varphi) \rho_0^{\bar r-r+1} + O(\rho_0)^{\bar r-r +2},
\end{eqnarray*}
we obtain
\begin{equation}\label{expansion-int-K}
\int_{\hat\gamma_{\rho_0}} \hat K := \int_{0}^{2 \pi} \hat K(\varphi, \rho(\varphi; \rho_0)) \, d \varphi = \sum_{i \geq \bar{r}-r} \beta_i \, \rho_0^i,
\end{equation}
where in the last equality we have performed integration term-by-term of the convergent power series at $\rho_0=0$ of the analytic function $\hat K(\varphi, \rho(\varphi; \rho_0))$, which is justified by \eqref{log-expansion}.

Comparing the coefficients in both series \eqref{log-expansion} and \eqref{expansion-int-K} and taking into account the property \eqref{relation-b-eta} the statements (i) and (ii) of the theorem follows.
\end{proof}

\begin{corollary}\label{Cor-centers-Mopq}
Under the hypothesis of Theorem \ref{Th-centers-Mopq}, the first coefficients $\beta_i$ in the expansion \eqref{Teo-int-K-alphas} are given by
\begin{eqnarray*}
\beta_{\bar{r}-r} &=& \int_{0}^{2 \pi} \frac{D(\varphi) a_1^{\bar{r}-r}(\varphi) K_{\bar{r}}(\cos\varphi, \sin\varphi)}{G_r(\varphi)}  \, d \varphi,  \\
\beta_{\bar{r}-r+1} &=& \int_{0}^{2 \pi} D(\varphi) \left( \frac{(\bar{r}-1) a_1^{\bar{r}-r-1}(\varphi) a_2(\varphi) K_{\bar{r}}(\cos\varphi, \sin\varphi)}{G_r(\varphi)} + \right. \\
 & & \left. \frac{a_1^{\bar{r}-r+1}(\varphi) (K_{\bar{r}+1}(\cos\varphi, \sin\varphi) G_r(\varphi) - K_{\bar{r}}(\cos\varphi, \sin\varphi) G_{r+1}(\varphi))}{G_r^2(\varphi)}  \right) \, d \varphi.
\end{eqnarray*}
being $a_1 : [0, 2 \pi] \to \mathbb{R}$ is defined in \eqref{a1-eq}.
\end{corollary}

\begin{proof}
It is straightforward from the proof of Theorem \ref{Th-centers-pq-period} taking into account that {\small
\begin{eqnarray*}
\hat{K}(\varphi, \rho) &=&  D(\varphi) \left(\frac{K_{\bar{r}}(\cos\varphi, \sin\varphi)}{G_r(\varphi)} \rho^{\bar{r}-r} + \right. \\
 & & \left. \frac{K_{\bar{r}+1}(\cos\varphi, \sin\varphi) G_r(\varphi) - K_{\bar{r}}(\cos\varphi, \sin\varphi) G_{r+1}(\varphi)}{G_r^2(\varphi)} \rho^{\bar{r}-r+1} +  O(\rho^{\bar{r}-r+2})\right),
\end{eqnarray*}  }
and that the first explicit coefficients $\beta_{\bar{r}-r}$ and $\beta_{\bar{r}-r+1}$ in the series \eqref{expansion-int-K} are just the expressions given in the corollary.
\end{proof}

\begin{corollary}\label{Cor-centers-pq-period}
Under the assumptions of Theorem \ref{Th-centers-Mopq} applied to a family $\mathcal{X}_\lambda$ of vector fields with parameters $\lambda \in \Lambda$, if $K_{\bar{r}}(\cos\varphi, \sin\varphi)$ is sign-defined and vanishes identically only when the parameters lie in $\hat\Lambda \subset \Lambda$  then the solution of the equation $\eta_{1}(\lambda) = 1$ if $\bar{r} - r = 0$ or $\eta_{\bar{r}-r+1}(\lambda) = 0$ if $\bar{r} - r \geq 1$ is $\lambda \in  \hat\Lambda$.
\end{corollary}
\begin{proof}
Taking into account that $D(\varphi) a_1(\varphi) G_r(\varphi) \neq 0$ for all $\varphi$, it is direct that the expression of $\beta_{\bar{r}-r}$ given in Corollary \ref{Cor-centers-Mopq} only vanishes for the parameters $\lambda \in  \hat\Lambda$. Then the corollary follows recalling statements (i) and (ii) of Theorem \ref{Th-centers-Mopq} according to whether $\bar{r} - r = 0$ or $\bar{r} - r \geq 1$, respectively.
\end{proof}

\begin{remark}
{\rm It is very important to stress that the called Bautin method can be applied to compute the Poincar\'e-Lyapunov quantities $\eta_i$ when $\mathcal{X} \in {\rm Mo}^{(p,q)}$. It consists in expanding the solution $\rho(\varphi; \rho_0)$ of the Cauchy problem (\ref{eq3*}) with initial condition $\rho(0; \rho_0) = \rho_0$ as $\rho(\varphi; \rho_0) = \sum_{i \geq 1} a_i(\varphi) \rho_0^i$ with $a_1(0)=1$ and $a_i(0) = 0$ for all $i \geq 2$ and the function $\mathcal{F}$ appearing in (\ref{eq3*}) of the form $\mathcal{F}(\varphi, \rho)= \sum_{i \geq 1} \mathcal{F}_i(\varphi) \rho^i$ with $\mathcal{F}_1 = F_r / G_r$. Some computations reveal that one can find the functions $a_i(\varphi)$ as the solutions of linear Cauchy problems. As a taste, the first of these equations are
\begin{eqnarray*}
a'_1(\varphi) &=& a_1(\varphi) \mathcal{F}_1(\varphi), \\
a'_2(\varphi) &=& a_2(\varphi) \mathcal{F}_1(\varphi) + a_1^2(\varphi) \mathcal{F}_2(\varphi).
\end{eqnarray*}
Finally the Poincar\'e map is $\Pi(\rho_0) = \sum_{i \geq 1} \eta_i \rho_0^i$ with $\eta_i = a_i(2 \pi)$. Even though Bautin's approach is theoretically well established, the closed form of the functions $a_i(\varphi)$ can be very hard to obtain in applications, even impossible due to the involved quadrature. This happens in several of the forthcoming analyzed examples. Is in this context where Theorem \ref{Th-centers-Mopq} and Corollary \ref{Cor-centers-Mopq} applies successfully showing that, without the explicit expressions of $a_i(\varphi)$ we can deduce that $a_i(2 \pi) = 0$ in some cases. Moreover, they also show how to construct $\eta_{\bar{r}-r+j}$ with the sole knowledge of $a_i(\varphi)$ for $i=1, \ldots, j+1$, see Corollary \ref{Cor-centers-Mopq} for example.  }
\end{remark}

\section{Examples} \label{S8}

In this section we present some non-trivial examples illustrating the application of the results obtained in this work.

\subsection{Quasihomogeneous vector field}

We analyze the following example:
\begin{equation}\label{ejemplo-suficientCC}
\dot{x} = A x^3 + B y, \ \ \ \dot{y} = C x^5 + D x^2 y.
\end{equation}
This is a $(1,3)$-quasihomogeneous vector field of degree $r=2$ and consequently $v(x,y) = C x^6 - 3 A x^3 y + D x^3
y - 3 B y^2$ is an inverse integrating factor of \eqref{ejemplo-suficientCC}. The center-focus problem has been solved
in \cite{AFG}, being the outcome that the monodromic condition is $\Delta := (D-3 A)^2 + 12 B C < 0$ while the necessary and sufficient center
condition is additionally the restriction $3 A + D = 0$.
\newline

An analysis of the leading term of the invariant branches at the origin gives that $\alpha_0$ must be a root of the quadratic polynomial $\mathcal{Q}(\eta) = -C + (3 A - D) \eta + 3 B \eta^2$ with discriminant $\Delta$ and the leading exponent is $3$. Therefore $y^*(x)$ and $\bar{y}^*(x)$ are invariant branches with $y^*(x) = \alpha_0  x^3 + \sum_{i \geq 1} \alpha_i x^{3+i/n}$ for some index $n$ and where $\alpha_0 = (D-3 A + \sqrt{\Delta})/ (6 B)$ and certain coefficients $\alpha_i \in \mathbb{C}$. Recall that $B \neq 0$ under monodromy. Here we easily check (therefore as a by product we prove that $\alpha_i = 0$ for all $i>1$) that using only the leading terms of the branches we get
\begin{eqnarray*}
F(x,y) &=& (y - y^*(x))(y-\bar{y}^*(x)) =  (y -\alpha_0 x^3)(y- \bar\alpha_0 x^3 ) \\
 &=& F_6(x,y) = y^2 + \left( \frac{A}{B} - \frac{D}{3 B}\right) x^3 y  -  \frac{C}{3 B} x^6,
\end{eqnarray*}
which is a $(1,3)$-quasihomogeneous polynomial inverse integrating factor of degree $s=6$ and the cofactor is $K(x,y) = K_2(x,y) = (3 A + D) x^2$,
therefore $\bar r = 2$.

Using $(1, 3)$-weighted polar coordinates we obtain a differential system \eqref{system-w-polares} but the closed form expression of $a_1(\varphi)$ is very difficult to obtain since we are not able to obtain a primitive of $F_2(\varphi)/G_2(\varphi)$. Therefore we do not know the value of $\eta_1 = a_1(2 \pi)$. Anyway, under the monodromic condition $\Delta < 0$ system \eqref{ejemplo-suficientCC} belongs to the ${\rm Mo}^{(1,3)}$ class. Additionally, $F(0,0)=0$ and $K_{2}$ is sign-defined so we may apply Corollary \ref{Cor-centers-pq-period} to conclude that $\eta_{1} = 1$ if and only if $3 A + D = 0$. Observe that we can reobtain that the origin of system \eqref{system-w-polares} is a focus when $3 A + D \neq 0$ just by using Corollary \ref{Corol-semidef}.

\subsection{Nilpotent cubic vector field}

In 1953 Andreev (\cite{Andreev1953}) showed that the origin of family
\begin{equation}\label{nil-cubic}
\dot x =  y + A x^2 y + B x y^2 + C y^3,  \ \ \ \dot y = -x^3 + P x^2 y + \kappa x y^2 + L y^3,
\end{equation}
is  monodromic and that it is a center if and only if $P = B + 3 L = (A + \kappa) L = 0$.

Since the origin of \eqref{nil-cubic} is a nilpotent singularity, it is known that $\mathbf{N}(\mathcal{X})$ has only one edge with associated weight $(p,q) = (1, n_A)$ being $n_A$ the so-called Andreev number. In \eqref{nil-cubic} one has $n_A=2$ and the vector field has the $(1,2)$-quasihomogeneous decomposition $\mathcal{X} = \mathcal{X}_1 + \cdots$  with $\mathcal{X}_1 = y \partial_x - x^3 \partial_y$, hence the leading degree is $r=1$.

An analysis of the leading term of the invariant branches at the origin gives that $\alpha_0$ must be a root of $\mathcal{Q}(\eta) = 1 + 2 \eta^2$ and $k = 2$. Then $y^*(x)$ and $\bar{y}^*(x)$ are invariant branches with $y^*(x) = (i / \sqrt{2}) x^2 + \sum_{j \geq 1} \alpha_j x^{2+ \frac{j}{n}}$ for certain coefficients $\alpha_j \in \mathbb{C}$ and branch index $n$. In order to determine $n$ we follow the lines of Remark \ref{Demina-remark}. There is just one dominant balance given by $E_0[y(x), x] = y y' + x^3$. Now we calculate the formal G\^{a}teaux derivative of the dominant balance at $y(x) = \alpha_0 x^k = (i / \sqrt{2}) x^2$, that is,
\begin{eqnarray*}
\frac{\delta E_0}{\delta y}[\alpha_0 x^k] &=& \lim_{s \to 0} \frac{1}{2} x^{3 + j} (i \sqrt{2} (4 + j) + 2 (2 + j) s x^j) = V(j) x^{3+j},
\end{eqnarray*}
with $V(j) = i (4 + j) / \sqrt{2}$. Therefore the Fuchs index of the dominant balance is $j= -4 \not\in \mathbb{Q}^+ \backslash \mathbb{N}$ and then the branch index is $n = 1$. In consequence $y^*(x)$ is expressed as the power series $y^*(x) = (i / \sqrt{2}) x^2 + \sum_{j \geq 1} \alpha_j x^{2+j}$. Since we know by Theorem \ref{teo-F-exists} that there is one real analytic invariant curve $F^{\mathbb{R}}(x,y)=0$ of \eqref{nil-cubic}, it is clear that each invariant branch of $F^{\mathbb{R}}$ has a convergent Puiseux series and therefore we can always choose that $y^*(x) = (i / \sqrt{2}) x^2 + \sum_{j \geq 1} \alpha_j x^{2+j}$ and its complex conjugated $\bar{y}^*(x)$ are convergent power series which implies that $F(x,y) = (y-y^*(x)) (y- \bar{y}^*(x))$ is a real analytic function at the origin.

Straightforward computations show that $F$ and $K$ have the following $(1,2)$-quasiho\-mogeneous expansions
\begin{eqnarray*}
F(x,y) &=& F_4(x,y) + \cdots = \frac{x^4}{2} + y^2 + \cdots, \\
K(x,y) &=& K_2(x,y) + \cdots = \frac{4}{5} P x^2  + \cdots,
\end{eqnarray*}
and therefore $s=4$ and $\bar{r} = 2$. Using $(1,2)$-weighted polar coordinates we obtain a polar system \eqref{system-w-polares} with expression
$d \rho/d \tau = F_1(\varphi) \rho + O(\rho^2)$, $d \varphi/d \tau = G_1(\varphi) +  O(\rho)$, where $F_1(\varphi) = (1 - \cos^2\varphi) \cos\varphi \sin\varphi$ and $G_1(\varphi) = -\cos^4\varphi - 2 \sin^2\varphi$. Clearly $G_1$ does not vanish on $\mathbb{S}^1$ and therefore system \eqref{nil-cubic} lies in ${\rm Mo}^{(1,2)}$, as expected. We know that $\beta_0 = 0$ because $\bar{r} - r = 1$, see \eqref{Teo-int-K-alphas}. Hence $\eta_1 = 1$ by Theorem \ref{Th-centers-Mopq}(i).

From statement (ii) of Theorem \ref{Th-centers-Mopq} it follows that $\beta_1 = 4 \eta_{2}$. Taking into account the expression of $\beta_1$ given in Corollary \ref{Cor-centers-Mopq}, using Corollary \ref{Cor-centers-pq-period} with the sign-defined function $K_{2}(\cos\varphi, \sin\varphi) = \frac{4}{5} P \cos^2\varphi$  we conclude that $\eta_{2} = 0$ if and only if $P = 0$. Further computations show that $a_2(\varphi) \equiv 0$ provided $P=0$, in which case we compute $\beta_2 = 0$ and consequently $\eta_3=0$.
\newline

We finish by noticing that by Lemma \ref{lemma1-Mopq} we know that the explicit expression of $a_1(\varphi)$ is
$$
a_1(\varphi) = \frac{2^{3/4}}{(11 - 4 \cos(2 \varphi) +  \cos(4 \varphi))^{1/4}}.
$$
Then $\eta_1 = a_1(2 \pi) = 1$ confirming our previous result. In order to do further computations we need the expression of $a_2(\varphi)$ but to obtain it in closed form with arbitrary $P$ is very difficult.

\begin{remark}
{\rm Following the theory developed in \cite{De,DGV}, we may compute the Newton diagram associated to the differential equation of the orbits of \eqref{nil-cubic}. It is given by the convex hull of the set of points $\{ (-1,2), (1,2), (0,3), (-1,4),$ $(3,0),$ $(2,1)\}$. This Newton diagram has one unique descending segment with endpoints $(-1,2)$ and $(3,0)$ implying that the associated weights are $(p,q) = (1, 2)$.}
\end{remark}

\subsection{A vector field in the class ${\rm Mo}^{(p,q)}$}

In \cite{AGG2} it is proved that system
\begin{equation}\label{Algaba-JDE-example}
\dot{x} = y^3 + 2 a x^3 y+2 x (\alpha x^4 + \beta x y^2), \ \ \dot{y} = -x^5- 3 a x^2 y^2 + 3 y (\alpha x^4 +
\beta x y^2)
\end{equation}
with $\alpha \beta \neq 0$ is not orbitally reversible nor formally integrable but, under the monodromy condition
\begin{equation}\label{mon-AlgabaJDE}
0 \neq |a| < 1/\sqrt{6},
\end{equation}
there are real values of $(\alpha, \beta, a)$ such that the origin is a center. We observe that \eqref{Algaba-JDE-example} with weights $(p,q) = (2,3) \in W(\mathbf{N}(\mathcal{X}))$ has the $(2,3)$-quasihomogeneous expansion $\mathcal{X} = \mathcal{X}_{r} + \cdots$ with $r=7$ and $\mathcal{X}_{7} = (y^3 + 2 a x^3 y) \partial_x + (-x^5- 3 a x^2 y^2) \partial_y$. Moreover it is easy to check that $\mathcal{X} \in {\rm Mo}^{(2,3)}$ because $\Omega_{23} = \emptyset$.
\newline

An analysis of the leading term of the invariant branches at the origin gives that $\alpha_0$ must be a root of the quartic polynomial $\mathcal{Q}(\eta) = 1 + 6 a \eta^2 + 3 \eta^4/2$ and the leading exponent is $q/p = 3/2$. We consider the set of invariant branches $\{ y_1^*(x), \bar{y}_1^*(x), y_2^*(x), \bar{y}_2^*(x) \}$.  We have $y_1^*(x) = \alpha_0  x^{3/2} + \sum_{i \geq 1} \alpha_i x^{\frac{3}{2}+\frac{i}{n}}$ where $n$ is an even integer and
\[
\alpha_0 = \sqrt{-2 a + \sqrt{4 a^2-2/3}}
\]
and certain coefficients $\alpha_i \in \mathbb{C}$.

To get the index $n$ of the branch $y_1^*$ we use Remark \ref{Demina-remark}. There is one dominant balance $E_0[y(x), x] = (y^3 + 2 a x^3 y) y' + (x^5 + 3 a x^2 y^2)$. The formal G\^{a}teaux derivative $E_0$ at $y(x) = \alpha_0  x^{3/2}$ is
\begin{eqnarray*}
\frac{\delta E_0}{\delta y}[\alpha_0  x^{3/2}] &=& V(j) x^{5+j},
\end{eqnarray*}
with $V(j) = \xi(a) (j+6)$ where $\xi(a) \neq 0$ under the restriction \eqref{mon-AlgabaJDE}. In short $j=-6 \not\in \mathbb{Q}^+ \backslash \mathbb{N}$ is the unique Fuchs index and then the branch index is $n = 2$. So we obtain the expansion $y_1^*(x) = \alpha_0  x^{3/2} + \sum_{i \geq 1} \alpha_i x^{\frac{3+i}{2}}$, and further computations reveal that $\alpha_i = 0$ for the first indexes $0 < i < 5$. Indeed we claim that $\alpha_i = 0$ for all $i >0$ and that this also happens for all he Puiseux series of the other branches.  The claim follows because we have that
\begin{eqnarray*}
F(x,y) &=& (y - y_1^*(x))(y-\bar{y}_1^*(x))(y - y_2^*(x))(y-\bar{y}_2^*(x)) \\
 &=& F_{12}(x,y) = \frac{2}{3} x^6 + 4 a x^3 y^2 + y^4, \\
K(x,y) &=& K_8(x,y) = 12 x (\alpha x^3 + \beta y^2),
\end{eqnarray*}
hence $s=12$ and $\bar{r} = 8$. Using $(2,3)$-weighted polar coordinates we obtain a system \eqref{system-w-polares} with expression $d \rho/d \tau = F_7(\varphi; a) \rho + O(\rho^2)$, $d \varphi/d \tau=  G_7(\varphi; a) + O(\rho)$, where
\begin{eqnarray*}
F_7(\varphi; a) &=& \sin\varphi \cos\varphi (1 + 6 a \cos\varphi - 8 \cos(2 \varphi) + 10 a \cos(3 \varphi) - \cos(4 \varphi)) / 8, \\
G_7(\varphi; a) &=& -(2 \cos^6\varphi + 12 a \cos^3\varphi \sin^2\varphi + 3 \sin^4\varphi).
\end{eqnarray*}
It is easy to see that $G_7(\varphi; a) \neq 0$ for any $\varphi \in \mathbb{S}^1$ and $a$ satisfying the monodromic condition \eqref{mon-AlgabaJDE} so the origin of system \eqref{Algaba-JDE-example} belongs to the monodromic class ${\rm Mo}^{(2,3)}$. Moreover $\beta_0 = 0$ because $\bar{r} - r = 1$ by \eqref{Teo-int-K-alphas} and then $\eta_1 = 1$ from statement (ii) of Theorem \ref{Th-centers-Mopq}. We stress that we cannot use Corollary \ref{Cor-centers-pq-period} because $K_{8}(\cos\varphi, \sin\varphi)$ is not sign-defined.

From Lemma \ref{lemma1-Mopq} we calculate
$$
a_1(\varphi) = \exp\left( \int_0^\varphi \frac{F_7(\theta; a)}{G_7(\theta; a)} d \theta \right)  =  \frac{2^{1/12}}{(-G_7(\varphi; a))^{1/12}},
$$
confirming again that $\eta_1 = a_1(2 \pi) = 1$. Now we use statement (ii) of Theorem \ref{Th-centers-Mopq} and we compute $\beta_1$ with the formula given in Corollary \ref{Cor-centers-Mopq}:
$$
12 \eta_2 = \beta_{1} = \int_{0}^{2 \pi} \frac{D(\varphi) a_1(\varphi) K_{8}(\cos\varphi, \sin\varphi)}{G_7(\varphi; a)}  \, d \varphi = \alpha f(a) + \beta g(a)
$$
where
$$
f(a) = \int_{0}^{2 \pi} A(\varphi; a) \cos^3\varphi  \, d \varphi, \ \ g(a) = \int_{0}^{2 \pi} A(\varphi; a) \sin^2\varphi \, d \varphi,
$$
with $A(\varphi; a) = 6 \times 2^{1/12} \cos\varphi (\cos(2\varphi) -5)/(-G_7(\varphi; a))^{13/12}$.

\begin{remark}
{\rm Using the fact that $F^{13/12}$ is an inverse integrating factor of \eqref{Algaba-JDE-example}, in the work \cite{GaGi1} it is shown that the points in the parameter space that correspond to a center lie in the surface $\mathcal{C} = \{
(a, \alpha, \beta) \in \mathbb{R}^3 : |a| < 1/\sqrt{6}, \ \alpha f(a) + \beta g(a) = 0\}$ where $f$ and $g$ are analytic functions and $g(0)=0$. In particular $(a, \alpha, \beta) = (0,0,\beta) \in \mathcal{C}$ corresponds with a time-reversible center and $(a, \alpha, \beta) = (a,0, 0) \in \mathcal{C}$ when $|a| < 1/\sqrt{6}$ with the Hamiltonian centers. See the definition of time-reversible center in \cite{GM}.
 }
\end{remark}

\subsection{A vector field not lying in the class ${\rm Mo}^{(p,q)}$}

In the work \cite{GaGi2} it is analyzed the family of vector fields
\begin{eqnarray}
\dot{x} &=& \lambda_1 (x^6 + 3 y^2) (-y + \mu x) + \lambda_2 (x^2+y^2)(y + A x^3), \nonumber \\
\dot{y} &=& \lambda_1 (x^6 + 3 y^2) (x + \mu y) + \lambda_2 (x^2+y^2)(-x^5 + 3 A x^2 y ), \label{ejemplo3-DCCD}
\end{eqnarray}
with parameter space $(\lambda_1, \lambda_2, \mu, A) \in \mathbb{R}^4$. The origin is a degenerate singularity and there are members of the family with both the origin monodromic and not monodromic. In \cite{GaGi2} it is proved that the origin of family \eqref{ejemplo3-DCCD} is monodromic if and only if the parameters lie in
$$
\Lambda = \{ (\lambda_1, \lambda_2, \mu, A) \in \mathbb{R}^4 : 3 \lambda_1 - \lambda_2 > 0, \ \ \lambda_1 - \lambda_2 > 0\}.
$$
Moreover there it is also proved, thanks to a given inverse integrating factor of \eqref{ejemplo3-DCCD}, that if we restrict the family to the subset of the parameter space $\Lambda \backslash \Lambda^*$ where $\Lambda^* = \{ \alpha_{11} \geq 0, \ \alpha_{13} \geq 0, \ A \lambda_2 + \sqrt{\alpha_{11}} \leq 0, \  -\mu \lambda_1 + \sqrt{\alpha_{13}} \geq 0 \}$ with the definitions $\alpha_{11} := -3 \lambda_1^2 + 4 \lambda_1 \lambda_2 + (-1 + A^2) \lambda_2^2$ and $\alpha_{13} := (-3 + \mu^2) \lambda_1^2 + 4 \lambda_1 \lambda_2 - \lambda_2^2$, then the following holds:
\begin{itemize}
  \item[(i)] The Poincar\'e map is linear: $\Pi(r_0) = \eta_1 r_0$;
  \item[(ii)] The origin is a center if and only if
  \begin{equation}\label{center-cond-V-ej}
  3 \lambda_1 \mu +  \sqrt{3} A \lambda_2 = 0.
  \end{equation}
\end{itemize}
We note that the set $\Lambda \backslash \Lambda^*$ corresponds to those parameters such that $\Theta^{-1}(0) \backslash \{\rho = 0\} = \emptyset$ and therefore there are no curves on $C$ of zero angular speed.
\newline

Our intention is to continue the analysis of family \eqref{ejemplo3-DCCD} without the restriction $\Lambda \backslash \Lambda^*$ and without the help of its inverse integrating factor. In this way we can prove the following result.

\begin{proposition}\label{prop-ejemplo3-DCCD-*}
We consider the monodromic family \eqref{ejemplo3-DCCD} with parameters lying in the set
$$
\bar{\Lambda} = \{ (\lambda_1, \lambda_2, \mu, A) \in \mathbb{R}^4 : \lambda_1 >0, \lambda_2 < 0, \lambda_2/\lambda_1 \in \mathbb{Q}^- \} \subset \Lambda.
$$
Then $\bar{\Lambda} \not\subset \Lambda \backslash \Lambda^*$ and the following holds:
\begin{itemize}
  \item[(i)] If $A \mu < 0$ then the origin of family \eqref{ejemplo3-DCCD} is a focus;

  \item[(ii)] If $\mu = 0$ then the origin of family \eqref{ejemplo3-DCCD} is a focus or a center according to whether $A \neq 0$ or $A=0$, respectively.
\end{itemize}
\end{proposition}
\begin{proof}
The reader can see the correctness of $\bar{\Lambda} \not\subset \Lambda \backslash \Lambda^*$ just by checking that the points $(1, -1, A, \mu) \in \bar{\Lambda}$ but they only belong to $\Lambda^*$ if $A \geq 2 \sqrt{2}$ and $\mu \leq -2 \sqrt{2}$.
\newline

The Newton diagram of \eqref{ejemplo3-DCCD} consists of two exterior vertices at $(0, 4)$ and $(8, 0)$ provided $\lambda_2 - 3 \lambda_1 \neq 0$ and $\lambda_1 - \lambda_2 \neq 0$, respectively, and an inner vertex at $(2, 2)$ if $\lambda_1^2+\lambda_2^2 \neq 0$. Therefore the weights $(p,q) \in W(\mathbf{N}(\mathcal{X}))$ for \eqref{campo-X} are $(p,q) = (1,1)$ and $(p,q) = (1, 3)$ whose leadings parts are $\mathcal{X}_2 = * \partial_x + \lambda_1 3 y^2 (x + y \mu) \partial_y$ and $\mathcal{X}_4 = \lambda_2 x^2 (A x^3 + y) \partial_x + * \partial_y$, respectively. Consequently, \eqref{ejemplo3-DCCD} is not in ${\rm Mo}^{(p,q)}$ since the leading parts have invariant lines through the origin.

An analysis of the leading term of the invariant branches at the origin gives that either $\alpha_0$ is a root of the polynomial $\mathcal{Q}_1(\eta) = 1 + \eta^2$ with leading exponent $1$ or $\alpha_0$ is a root of the polynomial $\mathcal{Q}_2(\eta) = 1 + 3 \eta^2$ with leading exponent $3$. In this example we see that, using only the leading terms of the branches, there are two invariant curves
\begin{eqnarray*}
F_1(x,y) &=& (y - y_1^*(x))(y-\bar{y}_1^*(x) = x^2+y^2, \\
F_2(x,y) &=& (y - y_2^*(x))(y-\bar{y}_2^*(x)) = y^2+x^6/3.
\end{eqnarray*}
This implies that all the coefficients in the branch Puiseux expansion are zero except the leading coefficient $\alpha_0$. The associated cofactors are
\begin{eqnarray*}
K^{(1)}(x,y) &=& 2 (\lambda_2 x y (1 - x^4) + A \lambda_2 x^2 (x^2 + 3 y^2) + \lambda_1 \mu (x^6 +  3 y^2), \\
K^{(2)}(x,y) &=& 6 (\lambda_1 x y (1 - x^4) + A \lambda_2 (x^4 + x^2 y^2) + \lambda_1 \mu (x^6 + y^2)).
\end{eqnarray*}
Following the idea of Remark \ref{multiples-F}, we take the cofactor $K = m_1 K^{(1)} + m_2 K^{(2)}$ of $F = F_1^{m_1} F_2^{m_2}$ with arbitrary $m_i \in \mathbb{Z}^+$. The $(1,1)$-quasihomogeneous expansion of $K$ is $K(x,y) = K_2(x,y) + K_4(x,y) + K_6(x,y)$, where $K_2(x,y) = 2 y ((3 m_2 \lambda_1 + m_1 \lambda_2) x +  3 (m_1 + m_2) \lambda_1 \mu y)$. The function $K_2$ is sign-defined in a neighborhood of the origin if and only if   $\lambda_2 = -3 \lambda_1 m_2/m_1$, hence $\lambda_2/\lambda_1 \in \mathbb{Q}^-$ because $\lambda_1 \neq 0$ by monodromy. Using this value of $\lambda_2$ we have
\begin{eqnarray*}
K_2(x,y)  &=& 6 (m_1 + m_2) \lambda_1 \mu  y^2, \\
K_4(x,y) &=& - \frac{6 m_2 A \lambda_1}{m_1} ((m_1 + 3 m_2) x^2 + 3 (m_1 + m_2) y^2) x^2, \\
K_6(x,y)  &=& 2 (m_1 + 3 m_2) \lambda_1 \mu  x^6.
\end{eqnarray*}
Therefore $K$ is positive (or negative) defined in a punctured neighborhood of the origin if and only if $A \mu < 0$ and, in this case, the origin is a focus by Corollary \ref{Corol-semidef} and this proves the statement (i) of the proposition.

We continue assuming the same assumptions but now restricting to $\mu = 0$ (recall that $\lambda_1 \neq 0$). In this case $K(x,y) = K_4(x,y)$ that is sign-defined provided that $A \neq 0$, hence the origin is a focus by Corollary \ref{Corol-semidef}. When $A=0$ then $K(x,y) \equiv 0$ and $F$ is an analytic first integral so that the origin is a center and this proves statement (ii) of the proposition.
\end{proof}

\subsection{A family of vector fields not in ${\rm Mo}^{(p,q)}$ but in $S_{k \omega}$}

\begin{remark}\label{Skw}
{\rm In the paper \cite{Ga-Ma-Ma} the monodromic class $S_{k \omega}$ was introduced using the classical polar coordinates (with weights $(p,q)=(1,1)$ although may be they are not in $W(\mathbf{N}(\mathcal{X}))$). The class $S_{k \omega}$ is defined as the subset of monodromic vector fields with $G_r \not\equiv 0$, $G_r(\varphi) \geq 0$ (that is the flow rotates counterclockwise) such that for any $\varphi^* \in \Omega_{11}$ the following holds:
\begin{enumerate}
  \item[(a)] $F_r(\varphi^*) = 0$;
  \item[(b)] $G_{r+1}(\varphi^*) = 0$;
  \item[(c)] $[(G''_{r}-2 F'_{r}) G''_{r}](\varphi^*) > 0$;
  \item[(d)] $[(G'_{r+1}- F_{r+1})^2 -2 (G''_{r} - 2 F'_{r}) G_{r+2}](\varphi^*) < 0$.
\end{enumerate}
In \cite{Ga-Ma-Ma} it is proved that, for any $\mathcal{X} \in S_{k \omega}$, the Poincar\'e return map has the form $\Pi(x) = \eta_1 x + o(x)$ with
$$
\eta_1 = PV \int_0^{2 \pi} \frac{F_r(\varphi)}{G_r(\varphi)} d \varphi.
$$ }
\end{remark}

Also the authors of \cite{Ga-Ma-Ma} prove that the origin of family
\begin{equation}\label{ejemplo3-JDE-Armengol}
\dot{x} = b x^2 y + a x y^2 - b y^3 - x^4, \ \ \ \dot{y} = 4 b x y^2 + a y^3 + 2 x^5,
\end{equation}
is monodromic  if and only if the family is restricted to the parameter space
\begin{equation}\label{monodromy-ArmengolJDE}
\Lambda = \{ (a, b) \in \mathbb{R}^2 : b > 1/6 \}.
\end{equation}
Indeed family \eqref{ejemplo3-JDE-Armengol} belongs to the class $S_{k\omega}$, see Remark \ref{Skw}. In  \cite{Ga-Ma-Ma} it is proved that
$$
\eta_1 = \exp\left( \frac{2 \pi \sqrt{3}}{3} \frac{a}{b} \right)
$$
so that, under the monodromic restriction \eqref{monodromy-ArmengolJDE}, the origin is an attractor (resp. repeller) focus if $a < 0$ (resp. $a > 0$). Additionally, when $a=0$ the origin is a time-reversible center.
\newline

The origin is a degenerate singularity and the Newton diagram $\mathbf{N}(\mathcal{X})$ is the convex hull of the set of points $\{ (2, 2), (1, 3), (0, 4), (4, 1), (6, 0) \}$ which has two descending segments: one segment with endpoints $(0,4)$ and $(2,2)$ containing the point $(1,3)$ and another segment with endpoints $(2,2)$ and $(6,0)$ containing the point $(4,1)$. The weights associated to $\mathbf{N}(\mathcal{X})$ are $(p,q) = (1,1)$ and $(p,q) = (1, 2)$, respectively.

\subsubsection{Analysis in the edge of $\mathbf{N}(\mathcal{X})$ with weights $(p,q)=(1,1)$}

The $(1,1)$-quasihomo\-geneous expansion of \eqref{ejemplo3-JDE-Armengol} is $\mathcal{X} = \mathcal{X}_{2} + \cdots$ with $\mathcal{X}_{2} = (b x^2 y + a x y^2 -b y^3 ) \partial_x + (4 b x y^2 + a y^3) \partial_y$ which is not monodromic because $\mathcal{X}_{2}$ possesses the invariant line $y=0$, in agreement with Remark \ref{remark-Mopq}. An analysis of the leading term of the invariant branches at the origin gives that the leading coefficient  $\alpha_0$ is a root of the polynomial $\mathcal{Q}_1(\eta) = 3 + \eta^2$ with leading exponent $q/p = 1$. Therefore the invariant branches associated with this edge will have a Puiseux expansion of the form $y_1^*(x) = \alpha_0  x + \sum_{i \geq 1} \alpha_i x^{1+\frac{i}{n_1}}$ where $n_1 \in \mathbb{N}$, $\alpha_0 = i \sqrt{3}$ (a root of $\mathcal{Q}_1$) and certain coefficients $\alpha_i \in \mathbb{C}$. We are going to compute the index $n_1$ according to Remark \ref{Demina-remark}. The dominant balance associated to this edge is $E_0[y(x), x] = (b x^2 y + a x y^2 -b y^3 ) y' - (4 b x y^2 + a y^3)$ and  its formal G\^{a}teaux derivative at $\alpha_0  x$ is $\frac{\delta E_0}{\delta y}[\alpha_0  x] = V(j) x^{3+j}$ with $V(j) = -3 a j + 2 i \sqrt{3} b (3 + 2 j)$. The Fuchs index is $j= 6 i b /(\sqrt{3} a - 4 i b)$. It is easy to see that $j \in \mathbb{Q}$ only when $a=0$ (reversible center case) in which case $j=-3/2$. Therefore in all the parameter space
$j \not\in \mathbb{Q}^+ \backslash \mathbb{N}$ so that the branch index is $n_1 = 1$ and we obtain the power series $y_1^*(x) = \alpha_0  x + \sum_{i \geq 1} \alpha_i x^{1+i}$. In this case we define $G(x,y) = (y - y_1^*(x))(y-\bar{y}_1^*(x)) \in \mathbb{R}[[x,y]]$ which reduces in $\mathbb{C}[[x,y]]$ but it is an irreducible polynomial in $\mathbb{R}[[x,y]]$, and we use the results of \cite{AGR3} explained in section \ref{subsec-equiv-class} to conclude that there is a real formal unit $U(x,y)$ such that $F(x,y) = U(x,y) G(x,y) = 0$ is a real analytic invariant curve of \eqref{ejemplo3-JDE-Armengol}. In particular, the $(1, 1)$-quasihomogeneous expansion of $F$ and $G$ share the leading $(1,1)$-quasihomogeneous term which can be computed giving $G(x,y) = G_2(x,y) + \cdots = 3 x^2 + y^2 + \cdots$. Therefore we obtain the $(1, 1)$-quasihomogeneous expansion of $F$ and its cofactor $K$:
\begin{eqnarray*}
F(x,y) &=& F_2(x,y) + \cdots = 3 x^2 + y^2 + \cdots, \\
K(x,y) &=& K_2(x,y) + \cdots = 2 y (b x + a y) + \cdots,
\end{eqnarray*}
so that $s= \bar r = r = 2$. But we cannot apply Corollary \ref{Corol-semidef} because we do not know the explicit expression of $K$ and $K_2$ is not sign-defined for parameters in $\Lambda$.

\subsubsection{Analysis in the edge of $\mathbf{N}(\mathcal{X})$ with weights $(p,q)=(1,2)$}

The $(1,2)$-quasihomo\-geneous expansion of \eqref{ejemplo3-JDE-Armengol} is $\mathcal{X} = \mathcal{X}_{3} + \cdots$ with $\mathcal{X}_{3} = (b x^2 y - x^4) \partial_x + (4 b x y^2 + 2 x^5) \partial_y$. Analyzing the leading term of the invariant branches at the origin one can see that $\alpha_0$ is a root of the polynomial $\mathcal{Q}_2(\eta) = 1 + \eta + b \eta^2$  with leading exponent $q/p=2$. The invariant branches of this edge have the Puiseux expansion $y_2^*(x) = \alpha_0  x^2 + \sum_{i \geq 1} \alpha_i x^{2+\frac{i}{n_2}}$ where $n_2 \in \mathbb{N}$, $\alpha_0 = (-1 + \sqrt{1 - 4 b})/(2 b)$ (a root of $\mathcal{Q}_2$) and some  coefficients $\alpha_i \in \mathbb{C}$. The calculation of $n_2$ requires to consider the dominant balance $E_0[y(x), x] = (b x^2 y - x^4) y' - (4 b x y^2 + 2 x^5) = 0$ whose formal G\^{a}teaux derivative at $\alpha_0 x^2$ is $\frac{\delta E_0}{\delta y}[\alpha_0  x^2] = V(j) x^{5+j}$ with $V(j) = [-4 \sqrt{1 - 4 b} + (-3 + \sqrt{1 - 4 b}) j]/2$. The Fuchs index is $j= (-1 - 3 \sqrt{1 - 4 b} + 4 b)/(2 + b)$ and therefore $j = s_1/s_2 \in \mathbb{Q}^+ \backslash \mathbb{N}$ when $b = -2 (s_1 - s_2) (s_1 + 2 s_2)/(s_1 - 4 s_2)^2$, and there are infinitely many pairs $(s_1, s_2) \in \mathbb{N}^2$ giving a monodromic value $b$ according to \eqref{monodromy-ArmengolJDE}. Then we cannot use Remark \ref{Demina-remark} to compute the index $n_2$, rather we need to apply the general theory developed in \cite{Br,Bru,De}.

Anyway we can overcome this difficulty by using the ideas of Remark \ref{V-Xr}. The irreducible factors of $F_s$ are factors of the inverse integrating factor $V(x,y) = (x, 2 y) \wedge \mathcal{X}_{3} = x^2 f(x,y)$ where $f(x,y) = x^4 + x^2 y + b y^2$ is irreducible in $\mathbb{R}[x,y]$. Moreover, taking into account the leading terms of the branches we see that $(y - \alpha_0  x^2)(y - \bar\alpha_0 x^2) = f(x,y)/b$ where $\bar\alpha_0 = (-1 - \sqrt{1 - 4 b})/(2 b)$. This computations shows that $s=4$ and $F_4(x,y) = f(x,y)$ whose cofactor is $K_3(x,y) = -2 x^3 + 8 b x y$.

Since $K(x,y) = K_3(x,y) + \cdots$ it follows that $\mathcal{J}^2K(x,y) = 8 b x y + c(a,b) y^2$ for some function $c(a,b)$. Clearly $\mathcal{J}^2K$ is not sign-defined since $b \neq 0$ in $\Lambda$ and therefore we cannot use Corollary \ref{Corol-semidef} because we do not know the hole expression of $K$.

\subsection{Family not lying in the monodromic classes ${\rm Mo}^{(p,q)}$ and $S_{k \omega}$}

In the work \cite{AGR} it is proved that family
\begin{equation}\label{ejemplo-ultimo}
\dot{x} = y (a x^2 + b x y + c y^2), \ \ \dot{y} = y^2 (a x + b y) + d x^5
\end{equation}
has a monodromic singularity at the origin if and only if the parameters lie in $\Lambda = \Lambda_1 \cup \Lambda_2$ with
\begin{eqnarray*}
\Lambda_1 &=& \{ (a, b, c, d) \in \mathbb{R}^4 : c d < 0, a = b = 0 \}, \\
\Lambda_2 &=& \{ (a, b, c, d) \in \mathbb{R}^4 : c d < 0, a c > 0 \}.
\end{eqnarray*}
The set $\Lambda_1$ corresponds to the Hamiltonian stratum, hence the origin is a center. We notice that the family never belongs to the class $S_{k \omega}$, see Remark \ref{Skw}. We continue focusing only in the monodromic parameter space $\Lambda_2$.

In this family $\mathbf{N}(\mathcal{X})$ has two segment and the associated weights are $W(\mathbf{N}(\mathcal{X}))$ $= \{ (1, 1), (1, 2)\}$.

\subsubsection{Analysis in the edge of $\mathbf{N}(\mathcal{X})$ with weights $(p,q)=(1,2)$} \label{subsec-ejSkw}

We get the $(1,2)$-quasihomogeneous expansion  $\mathcal{X} = \mathcal{X}_3 + \cdots$ with $\mathcal{X}_3 = a x^2 y \partial_x + x (d x^4 + a y^2) \partial_y$, hence $r=3$ and $\mathcal{X} \not\in {\rm Mo}^{(1,2)}$ since $x=0$ is an invariant line of $\mathcal{X}_3$. The possible invariant branches of $\mathcal{X}$ at the origin are $y_j^*(x) = \alpha_0 x^2 + \sum_{i \geq 1} \alpha_i x^{2 +\frac{i}{n_j}}$ with $j=1,2$, being $\alpha_0$ a root of $\mathcal{Q}(\eta) = a \eta^2 - d$ and $n_j$ the index of the branch.

Let $F(x,y)=0$ be an irreducible real analytic invariant curve of $\mathcal{X}$ passing through the origin. Since $F(x,y) \neq y - y_1^*(x)$ because the invariant curve must have an isolated point at the origin it follows that it must be $F(x,y) = (y - y_1^*(x))(y-y_2^*(x))$. Therefore $\mathcal{P}=\mathcal{Q}$ where $\mathcal{P}(\eta)$ is the determining polynomial of $F$. In particular both branches $y_j^*(x)$ are simple and therefore $n_j=1$, see Remark \ref{simple+rotation-remark}. The same result can be achieved computing the Fuch's index of the associated dominant balance which is $-2 \not\in \mathbb{Q}^+ \backslash \mathbb{N}$. Recall that $d/a < 0$ by monodromy, hence $\alpha_0$ is purely imaginary, $y_2^*(x) =\bar y_1^*(x)$, and $F(x,y) = (y - y_1^*(x))(y-\bar{y}_1^*(x))$. Some computations reveal the $(1,2)$-quasihomogeneous expansions
\begin{eqnarray*}
F(x,y) &=& F_4(x,y) + \cdots = y^2 - \frac{d}{a} x^4 + \cdots, \\
K(x,y) &=& K_3(x,y) + \cdots = 2 a x y + \cdots.
\end{eqnarray*}
Clearly $K_3(x,y) \not\equiv 0$ in $\Lambda_2$ and the function $K_3$ is not sign-defined so we have no conclusion on the nature of the monodromic point because we cannot use Corollary \ref{Corol-semidef}.

\subsubsection{Analysis in the edge of $\mathbf{N}(\mathcal{X})$ with weights $(p,q)=(1,1)$}

We get $\mathcal{X} = \mathcal{X}_2 + \cdots$ with $\mathcal{X}_2 = y (a x^2 + b x y + c y^2) \partial_x + y^2(a x + b y) \partial_y$, hence $r=2$ and $\mathcal{X} \not\in {\rm Mo}^{(1,1)}$. There is no branch associated to this edge because the non-zero leading coefficient would be a root of the polynomial $\mathcal{Q}(\eta) = c \eta^4$ which is impossible.

Using Remark \ref{V-Xr} we compute first the inverse integrating factor $V(x,y) = (x, y) \wedge \mathcal{X}_{2} = -c y^4$. Then $F_s(x,y) = y^s$ and its cofactor $K_2(x,y) = s y (a x + b y) \not\equiv 0$. The function $K_2$ is not sign-defined for parameters in $\Lambda_2$.

\subsection{Ma\~{n}osa monodromic family I}

The simplest case of Example B in the work \cite{Ma} is the family
\begin{equation}\label{ejemplo-manyosa}
\dot{x} = x y^2 - y^3 + a x^5, \ \ \dot{y} = 2 x^7 - x^4 y + 4 x y^2 + y^3,
\end{equation}
which has a monodromic singularity at the origin with parameters $\Lambda = \{a \in \mathbb{R} : \Delta(a) := 32 - (1+3 a)^2 >0 \}$. In \cite{Ma} it is proved that, restricting \eqref{ejemplo-manyosa} at $\Lambda$,
\begin{equation}\label{eta1-manyosa}
\eta_1 = \exp\left( \pi + \frac{4 \pi a}{\sqrt{\Delta(a)}} \right).
\end{equation}
From here it is straightforward to see that $\eta_1 \neq 1$ if $a \neq -31/25$ so that the origin is an stable focus when $-(4 \sqrt{2}+1)/3 < a < -31/25$ while it is an ustable focus when $-31/25 < a < (4 \sqrt{2}-1)/3$. We do not know the nature of the singularity when $a = -31/25$ because $\eta_1 = 1$.
\newline

Family \eqref{ejemplo-manyosa} has $\mathbf{N}(\mathcal{X})$ composed by two segment with associated weights  $W(\mathbf{N}(\mathcal{X}))$ $= \{ (1, 1), (1, 3)\}$.

\subsubsection{Analysis in the edge of $\mathbf{N}(\mathcal{X})$ with weights $(p,q)=(1,1)$}

We get the $(1,1)$-quasihomogeneous expansion  $\mathcal{X} = \mathcal{X}_2 + \cdots$ with $\mathcal{X}_2 = (x y^2 - y^3) \partial_x + (4 x y^2 + y^3) \partial_y$, hence $r=2$ and $\mathcal{X} \not\in {\rm Mo}^{(1,1)}$ since $y=0$ is an invariant line of $\mathcal{X}_2$. The possible invariant branches of $\mathcal{X}$ at the origin are $y_j^*(x) = \alpha_0 x + \sum_{i \geq 1} \alpha_i x^{1 +\frac{i}{n_j}}$ with $j=1,2$, being $\alpha_0 \neq 0$ a root of $\mathcal{Q}(\eta) = -\eta^2 (4 + \eta^2)$ and $n_j$ the index of the branch. Arguing like in subsection \ref{subsec-ejSkw} we conclude that $n_j=1$, $y_2^*(x) = \bar{y}_1^*(x)$ and that $F(x,y) = (y-y_1^*(x)) (y-\bar{y}_1^*(x)) = 0$ is a real analytic invariant curve of $\mathcal{X}$ having, together with its cofactor $K$, the $(1,1)$-quasihomogeneous expansions $F(x,y) = F_2(x,y) + \cdots = 4 x^2 + y^2 + \cdots$ and $K(x,y) = K_2(x,y) + K_4(x,y) + \cdots$ where
\begin{eqnarray*}
K_2(x,y) &=& 2 y^2, \\
K_4(x,y) &=& 2 a x^4 - \frac{6}{17} (1 + a) x^2 y^2.
\end{eqnarray*}
The polynomial $K_2$ is sign-defined but it is not positive or negative defined in a punctured neighborhood of the origin, hence we cannot still apply Corollary \ref{Corol-semidef}. Now we consider the polynomial $K_2 + K_4$ that can be easily analyzed since it has degree 2 in the variable $y$. We see that when $a >0$ then $K_2 + K_4$ is positive or negative defined in a punctured neighborhood of the origin which implies that $K$ is sign-defined in a neighborhood of the origin and therefore the origin of \eqref{ejemplo-manyosa} is a focus in $\{ a >0 \} \cap \Lambda$ by Corollary \ref{Corol-semidef}. With our method we are not able to prove that when $\{ a < 0 \} \cap \Lambda$ the origin is also a focus as it is proved in \cite{Ma}.

\begin{remark}\label{rem-NO-extension}
{\rm This is a good example to check the correctness of Remark \ref{re-no-exteder-r0}. Since $D(\varphi) \equiv 1$, we have
$$
PV \int_0^{2 \pi} \frac{D(\varphi) K_2(\cos\varphi, \sin\varphi)}{G_2(\varphi)} d \varphi  = \int_0^{2 \pi} \frac{4}{5 + 3 \cos(2 \varphi)} d \varphi  = 2 \pi > 0.
$$
which wrongly would imply that the focus is always unstable.}
\end{remark}

\subsubsection{Analysis in the edge of $\mathbf{N}(\mathcal{X})$ with weights $(p,q)=(1,3)$}

We get the $(1,3)$-quasihomogeneous expansion $\mathcal{X} = \mathcal{X}_4 + \cdots$ with $\mathcal{X}_4 = a x^5 \partial_x + (2 x^7 - x^4 y + 4 x y^2) \partial_y$, hence $r=4$ and $\mathcal{X} \not\in {\rm Mo}^{(1,3)}$ since $x=0$ is an invariant line of $\mathcal{X}_2$. The possible invariant branches of $\mathcal{X}$ at the origin are $y_j^*(x) = \alpha_0 x^3 + \sum_{i \geq 1} \alpha_i x^{3 +\frac{i}{n_j}}$ with $j=1,2$, being $\alpha_0 \neq 0$ a root of $\mathcal{Q}(\eta) = -2 + (1+3 a) \eta - 4 \eta^2$ and $n_j$ the index of the branch. First we note that the roots of $\mathcal{Q}$ are $(1 + 3 a \pm \sqrt{-\Delta(a)})/8$ and, since a necessary condition of monodromy is that $\alpha_0 \not\in \mathbb{R}$ we confirm again the definition of the set $\Lambda$. Arguing like in subsection \ref{subsec-ejSkw} we conclude that $n_j=1$, $y_2^*(x) = \bar{y}_1^*(x)$ and that $F(x,y) = (y-y_1^*(x)) (y-\bar{y}_1^*(x)) = 0$ is a real analytic invariant curve of $\mathcal{X}$ having, together with its cofactor $K$, the $(1, 3)$-quasihomogeneous expansions:
\begin{eqnarray*}
F(x,y) &=& F_6(x,y) + \cdots = \frac{1}{2} x^6 + y^2 - \frac{1}{4} x^3 (y + 3 a y) + \cdots, \\
K(x,y) &=& K_4(x,y) + \cdots = (3 a-1) x^4 + 8 x y + \cdots.
\end{eqnarray*}
From here we deduce that $\mathcal{J}^2K(x,y) = 8 x y$ is not sign-defined in a neighborhood of the origin and therefore we cannot use Corollary \ref{Corol-semidef}.

\subsection{Ma\~{n}osa monodromic family II}

In the Example C of the work \cite{Ma} it is proved that family
\begin{equation}\label{ejemplo-manyosa2}
\dot{x} = x^2 y - y^3 + a x^5 - x^{11}, \ \ \dot{y} = 4 x y^2 +9 x^7,
\end{equation}
has a monodromic singularity at the origin with parameters $\Lambda = \{a \in \mathbb{R} : |a| < 2 \}$ such that
$$
\eta_1 = \exp\left( \frac{10 \pi}{9}  \frac{a}{\sqrt{4-a^2}} \right)
$$
from where it is straightforward to see that the origin is an stable focus when $-2 < a < 0$ while it is an ustable focus when $0 < a < 2$. Also in \cite{Ma} it is proved that when $a=0$ the origin becomes a stable focus by using an adequate Lyapunov function.
\newline

Family \eqref{ejemplo-manyosa2} has $\mathbf{N}(\mathcal{X})$ composed by two segment with end points $(0,4)$ and $(2,2)$ the first one and $(2,2)$ and $(8,0)$ the second one. Therefore the associated weights are $W(\mathbf{N}(\mathcal{X})) = \{ (1, 1), (1, 3)\}$.

\subsubsection{Analysis in the edge of $\mathbf{N}(\mathcal{X})$ with weights $(p,q)=(1,1)$}

The $(1,1)$-quasihomo\-geneous expansion $\mathcal{X} = \mathcal{X}_2 + \cdots$ with $\mathcal{X}_2 = (x^2 y - y^3) \partial_x + 4 x y^2 \partial_y$ gives $r=2$ and $\mathcal{X} \not\in {\rm Mo}^{(1,1)}$ since $y=0$ is an invariant line of $\mathcal{X}_2$. The possible invariant branches of $\mathcal{X}$ at the origin are $y_j^*(x) = \alpha_0 x + \sum_{i \geq 1} \alpha_i x^{1 +\frac{i}{n_j}}$ with $j=1,2$, being $\alpha_0 \neq 0$ a root of $\mathcal{Q}(\eta) = -\eta^2 (3+ \eta^2)$ and $n_j$ the index of the branch. The associated Fuchs index is $j=-3/2 \not\in \mathbb{Q}^+ \backslash \mathbb{N}$, hence the index $n_j = 1$. We may apply Remark \ref{V-Xr} as follows. The inverse integrating factor of $\mathcal{X}_2$ is $V(x,y) = (x, y) \wedge \mathcal{X}_{2} = y^2 (3 x^2 + y^2)$. Then $F_s(x,y) = y^{m_1} (3 x^2 + y^2)^{m_2}$ with $s=m_1+2 m_2$ and $m_i \in \mathbb{N}$. The cofactor of $F_s$ is $K_2(x,y) = 2 (2 m_1 + m_2) x y \not\equiv 0$. Since the cofactor $K$ of $F$ is $K = K_2 + \cdots$ and  $K_2$ is not positive or negative defined we cannot use Corollary \ref{Corol-semidef}.

\subsubsection{Analysis in the edge of $\mathbf{N}(\mathcal{X})$ with weights $(p,q)=(1,3)$}

The $(1,3)$-quasihomo\-geneous expansion of $\mathcal{X}$ is $\mathcal{X} = \mathcal{X}_4 + \cdots$ with $\mathcal{X}_4 = (a x^5 + x^2 y) \partial_x + (9 x^7 + 4 x y^2) \partial_y$, hence $r=4$ and $\mathcal{X} \not\in {\rm Mo}^{(1,3)}$ since $x=0$ is an invariant line of $\mathcal{X}_4$. The possible invariant branches of $\mathcal{X}$ at the origin are $y_j^*(x) = \alpha_0 x^3 + \sum_{i \geq 1} \alpha_i x^{3 +\frac{i}{n_j}}$ with $j=1,2$, being $\alpha_0 \neq 0$ a root of $\mathcal{Q}(\eta) = 9 - 3 a \eta + \eta^2$ and $n_j$ the index of the branch. Since a necessary condition of monodromy is that $\alpha_0 \not\in \mathbb{R}$ we confirm again the definition of the set $\Lambda$. The Fuchs index associated to the branch $y_1^*$ with $\alpha_0 = 3(a + \sqrt{a^2-4})/2$ is $j =((6 \sqrt{a^2-4})/(5 a + 3 \sqrt{a^2-4}))$. It is easy to see that $j$ never belongs to $\mathbb{Q}^+$ when $a \in \Lambda$, hence we get $n_j=1$, $y_2^*(x) = \bar{y}_1^*(x)$ and that $F(x,y) = (y-y_1^*(x)) (y-\bar{y}_1^*(x)) = 0$ is a real analytic invariant curve of $\mathcal{X}$ with cofactor $K$ having the $(1, 3)$-quasihomogeneous expansions:
\begin{eqnarray*}
F(x,y) &=& F_6(x,y) + \cdots = 9 x^6 - 3 a x^3 y + y^2 + \cdots, \\
K(x,y) &=& K_4(x,y) + \cdots = 8 x y + 3 a x^4 + \cdots.
\end{eqnarray*}
Hence $\mathcal{J}^2K(x,y) = 8 x y$ is not sign-defined and so we cannot use Corollary \ref{Corol-semidef}.

\section{On the flow near the polycycle $\{\rho = 0 \}$} \label{S10}

We add this last section in order to see that, even though the Poincar\'e map $\Pi$ of the monodromic polycycle $\{\rho = 0 \}$ has a linear part because it is $\Pi(\rho_0) = \eta_1 \rho_0 + o(\rho_0)$ (this result was stated without proof in \cite{AA} and proved later in \cite{Me-2}), the flow $\rho(\varphi; \rho_0)$ of (\ref{eq3**}) may have no such liner part with respect to $\rho_0$. The next Proposition \ref{No-a1} gives a precise statement of the former claim.

\begin{proposition} \label{No-a1}
In general, when $\Omega_{pq} \neq \emptyset$, there is no function $a_1 : [0, 2 \pi] \subset \mathbb{R} \to \mathbb{R}$ such that flow $\rho(\varphi; \rho_0)$ of (\ref{eq3**}) can be expressed as $\rho(\varphi; \rho_0) = a_1(\varphi) \rho_0 + o(\rho_0)$ with $a_1(0)=1$ and $a_1(2 \pi) > 0$.
\end{proposition}

\subsection{Some consequences if $a_1(\varphi)$ would exists}

If $a_1$ would exists then it is straightforward to see that it satisfies the linear Cauchy problem
\begin{equation}\label{edo-a2}
G_r(\varphi) a'_1(\varphi) = a_1(\varphi) \, F_r(\varphi), \ \ a_1(0)=1.
\end{equation}
Since $\Omega_{pq} \neq \emptyset$ it may happen that the solution $a_1(\varphi)$ of the Cauchy problem \eqref{edo-a2} be divergent at $\varphi \in \Omega_{pq}$. We define the quantity $\xi_{pq}(\varphi)$ as the following Cauchy principal value
\begin{equation}\label{def-xi}
\xi_{pq}(\varphi) = PV \int_0^\varphi \frac{F_r(\theta)}{G_r(\theta)} d \theta,
\end{equation}
if it exist (meaning it takes a finite value). For any $\varphi \in \mathbb{S}^1$, we define the interval $I_\varepsilon(\varphi) = [0, \varphi] \backslash \cup_{i=1}^\ell (\varphi_i^*-\varepsilon, \varphi_i^*+\varepsilon)$ with $\Omega_{pq} = \{ \varphi^*_1, \ldots, \varphi^*_\ell \}$. Integrating \eqref{edo-a2} on $I_\varepsilon(\varphi)$ and taking limit when $\varepsilon \to 0$ we obtain that
$$
\log |a_1(\varphi)| = PV \int_0^\varphi \frac{a_1'(\theta)}{a_1(\theta)} d \theta = PV \int_0^\varphi \frac{F_r(\theta)}{G_r(\theta)} d \theta = \xi_{pq}(\varphi),
$$
for any $\varphi$ such that the principal value of the former right-hand side exists. Therefore $|a_1(\varphi)| = \exp\left( \xi_{pq}(\varphi) \right)$ when $\xi_{pq}(\varphi)$ exists. In particular, if $a_1$ exists then $\eta_1$ would be
\begin{equation}\label{edo-a3}
\eta_1 = \exp(\xi_{pq}(2 \pi)),
\end{equation}
provided this principal value exists.

\subsection{Proof of Proposition \ref{No-a1}}

The proof relies on the search of one example where $\xi_{pq}(2 \pi)$ exists but $\eta_1 \neq \exp(\xi_{pq}(2 \pi))$, hence \eqref{edo-a2} is wrong and therefore $a_1(\varphi)$ does not exist.

\subsubsection{Ma\={n}osa monodromic family I revisited}

We recall the family \eqref{ejemplo-manyosa} with monodromic parameters $\Lambda = \{a \in \mathbb{R} : \Delta(a) := 32 - (1+3 a)^2 >0 \}$. As we stated before, in \cite{Ma} the author prove that $\eta_1$ is given by expression \eqref{eta1-manyosa}. Since $(1,1) \in W(\mathbf{N}(\mathcal{X}))$ we use polar coordinates and we obtain the polar vector field $\dot{\rho} = F_2(\varphi) \rho + O(\rho^2) $, $\dot{\varphi} =  G_2(\varphi) + O(\rho)$ with $G_2(\varphi) = \sin^2\varphi (\sin^2\varphi+ 4 \cos^2\varphi)$ and $F_2(\varphi) = \sin^2\varphi (1 + 3 \sin\varphi \cos\varphi)$. Therefore
$$
\xi_{11}(2 \pi) = PV \int_0^{2 \pi} \frac{F_2(\varphi)}{G_2(\varphi)} d \varphi  = \int_0^{2 \pi} \frac{1 + 3 \sin\varphi \cos\varphi}{\sin^2\varphi+ 4 \cos^2\varphi} d \varphi = \pi.
$$
Therefore $\xi_{11}(2 \pi)$ exists but $\eta_1 \neq \exp(\xi_{11}(2 \pi))$.

\medskip

\noindent {\bf Acknowledgements.} The authors are partially supported by the Agencia Estatal de Investigaci\'on grant number PID2020-113758GB-I00 and an AGAUR grant number 2021SGR-01618.

%%%%%%%%%%%%%%%%%%%%%%%%%%%%%%%%%%%%%%%%%%%%%%%%%%%%%

\end{document}